\documentclass[11pt]{article}
\usepackage{amsfonts}
\usepackage{amsmath}

\newcommand{\be}{\begin{equation}}
\newcommand{\ee}{\end{equation}}
\newcommand{\bea}{\begin{eqnarray}}
\newcommand{\eea}{\end{eqnarray}}
\newcommand{\beas}{\begin{eqnarray*}}
\newcommand{\eeas}{\end{eqnarray*}}

\newtheorem{theorem}{Theorem}[section]
\newtheorem{definition}[theorem]{Definition}
\newtheorem{proposition}[theorem]{Proposition}
\newtheorem{corollary}[theorem]{Corollary}
\newtheorem{lemma}[theorem]{Lemma}
\newtheorem{remark}[theorem]{Remark}
\newtheorem{example}[theorem]{Example}
\newtheorem{examples}[theorem]{Examples}
\newtheorem{foo}[theorem]{Remarks}

\newenvironment{Remark}{\begin{remark}\rm}{\end{remark}}

\newenvironment{proof}{\addvspace{\medskipamount}\par\noindent{\it Proof}.}
{\unskip\nobreak\hfill$\Box$\par\addvspace{\medskipamount}}

\newcommand{\brak}[1]{\left(#1\right)}    
\newcommand{\crl}[1]{\left\{#1\right\}}   
\newcommand{\edg}[1]{\left[#1\right]}     

\newcommand{\p}{\mathbb{P}}

\newcommand{\E}[1]{{\mathbb{E}}\left[#1\right]}

\newcommand{\N}[1]{\left\|#1 \right\|}     

\DeclareMathOperator{\cs}{(\Sigma)}
\DeclareMathOperator{\cd}{(D)}
\DeclareMathOperator{\csd}{(\Sigma D)}

\title{Processes of class $\cs$, last passage times and drawdowns}

\author{Patrick Cheridito\thanks{Supported by NSF Grant DMS-0642361}\\
Princeton University\\ Princeton, NJ, USA
\and Ashkan Nikeghbali\\
Universit\"at Z\"urich\\
Z\"urich, Switzerland
\and Eckhard Platen\\ University of Technology\\ Sydney, Australia}

\date{October 2009}
\begin{document}

\maketitle

\begin{abstract}
We propose a general framework to study last passage times,
suprema and drawdowns of a large class of stochastic processes.
A central role in our approach is played by processes of class $\cs$.
After investigating convergence properties and a family of
transformations that leave processes of class $\cs$ invariant, we provide three
general representation results. The first one allows to recover
a process of class $\cs$ from its final value and the last time it
visited the origin. In many situations this gives access to the distribution
of the last time a stochastic process hit a certain level or was equal to its
running maximum. It also leads to a formula recently discovered by
Madan, Roynette and Yor expressing put option prices
in terms of last passage times. Our second representation result
is a stochastic integral representation of certain functionals of processes
of class $\cs$, and the third one gives a formula for their
conditional expectations. From the latter one can deduce the laws of a variety of
interesting random variables such as running maxima, drawdowns and
maximum drawdowns of suitably stopped processes.
As an application we discuss the pricing and hedging of options that depend
on the running maximum of an underlying price process and are
triggered when the underlying price drops to a given level or alternatively, when the
drawdown or relative drawdown of the underlying price attains a given height.\\[2mm]
{\bf Keywords:} Processes of class $\cs$, last passage times, drawdowns,
maximum drawdowns, relative drawdowns.
\end{abstract}

\noindent
{\bf Notation:}
Let $(\Omega, {\cal F}, ({\cal F}_t)_{t \ge 0}, \p)$ be a filtered
probability space satisfying the usual assumptions.
All stochastic processes $(X_t)$ will be indexed by $t \in \mathbb{R}_+$.
$\overline{X}_t$ denotes the running supremum $\sup_{u \le t} X_u$.
All semimartingales will be assumed to be c\`adl\`ag. Equalities and inequalities between
random variables are understood in the $\p$-almost sure sense.
We recall that a measurable process $(X_t)$ is said to be of class $\cd$ if
the family of random variables $\crl{|X_T| 1_{\crl{T<\infty}} : T \mbox{ a stopping time}}$
is uniformly integrable.

\setcounter{equation}{0}
\section{Introduction}

We propose a general framework to study various properties of continuous-time
stochastic processes which is based on the concept of
processes of class $\cs$. Non-negative local submartingales of class $\cs$ were
introduced by Yor \cite{yorinegal} and have been further studied by Nikeghbali
\cite{NEnl, NSub}. They are closely related to relative martingales
(see for instance, Az\'ema and Yor \cite{AYzeros} or Az\'ema et al. \cite{AMY}).
Compared to Yor \cite{yorinegal} and Nikeghbali \cite{NEnl, NSub}
we here work with the following larger class of stochastic processes:

\begin{definition}
We say a stochastic process $(X_t)$ is of class $\cs$ if it
decomposes as $X_t = N_t + A_t$, where\\[2mm]
\hspace*{5mm} {\rm (1)} $(N_t)$ is a c\`adl\`ag local martingale,\\[1mm]
\hspace*{5mm} {\rm (2)} $(A_t)$ is an adapted continuous finite variation
process starting at $0$,\\[1mm]
\hspace*{5mm} {\rm (3)} $\int_0^t 1_{\crl{X_u \neq 0}} dA_u = 0$ for all $t \ge 0$.\\[2mm]
We say $(X_t)$ is of class $\csd$ if it is of class $\cs$ and
of class $\cd$. By $L$ we denote the random time
$$
L := \sup \crl{t : X_t = 0} \quad \mbox{with the convention }
\sup \emptyset = 0.
$$
\end{definition}

According to this definition, every local martingale is of class $\cs$
and all uniformly integrable martingales are of class $\csd$.
In Lemma \ref{lemmasub} below it is shown that if a process $(X_t)$ of class $\cs$
is non-negative, then $(A_t)$ has to be increasing and $(X_t)$ is a local submartingale.
This case includes the absolute value $(|M_t|)$ of continuous local martingales
$(M_t)$ as well as drawdown processes $(\overline{M}_t - M_t)$ of local martingales
whose running suprema $(\overline{M}_t)$ are continuous. It will also follow from
Lemma \ref{lemmasub} that for every constant $K \in \mathbb{R}$, the process
$(K-M_t)^+$ is of class $\cs$ if $(M_t)$ is a local martingale with no positive jumps.
Many other processes, such as suitably transformed diffusions or the
Az\'ema submartingale in the filtration generated by the Brownian zeros,
fall into the class $\cs$. Lemma \ref{lemmaprod} below
shows that the product of processes of class $\cs$
with vanishing quadratic covariation is again of class $\cs$, and
in Lemma \ref{lemmatrans} it is proved that the class $\cs$
is stable under transformations of the form $(X_t) \mapsto (f(A_t)X_t)$, where
$f : \mathbb{R} \to \mathbb{R}$ is a locally bounded Borel function.

We show how to use It\^{o}'s formula and martingale techniques
to derive general results for processes of class $\cs$
with interesting consequences for a wide range of stochastic
processes and related random times. In particular, our arguments
do not need Markov or scaling properties. In Section \ref{secPrel} we
provide preliminary results on the convergence, positive parts and
products of processes of class $\cs$. Then we extend two results on
processes of the form $f(A_t)X_t$ from Nikeghbali \cite{NSub} to our setup.
In Section \ref{secRep} we prove three general representation results.
The first gives conditions under which a process $(X_t)$
of class $\cs$ converges to a limit $X_{\infty}$
almost everywhere on the set $\crl{L < \infty}$ and can be written as
\be \label{repXL}
X_t = \E{X_{\infty} 1_{\crl{L \le t}} \mid {\cal F}_t}, \quad t \ge 0.
\ee
Similarly to the case of uniformly integrable martingales, which are always of the form
$M_t = \E{M_{\infty} \mid {\cal F}_t}$, formula \eqref{repXL} represents $(X_t)$ in terms
of its final value and the last time it visited zero.
If $X_{\infty} = 1$, one can take expectations on both sides of \eqref{repXL} to obtain
$\E{X_t} = \p[L \le t]$. In situations where $\E{X_t}$ can be calculated, this give access to the
distribution of the random time $L$; see Nikeghbali and Platen \cite{NP} for examples.
If $(M_t)$ is a non-negative local martingale without positive jumps such that
$M_t \to 0$ almost surely, then for all $K \in \mathbb{R}_+$,
the submartingale $(K-M_t)^+$ is of class
$\cs$ with last zero $g_K = \sup \crl{t : M_t \ge K}$. So as a special case
of \eqref{repXL}, one obtains the following formula of Madan et al. \cite{MRY}:
\be \label{MRYform}
\E{(K-M_t)^+} = K \p[g_K \le t], \quad t \ge 0.
\ee
Our second representation result shows how to write random variables of the form
$h(A_{\infty})$ as stochastic integrals with respect to $(N_t)$, and the third one
gives a formula for $\E{h(A_{\infty}) \mid {\cal F}_T}$, where
$T$ is an arbitrary stopping time. This leads to closed form expressions for the
conditional distributions of $A_{\infty}$.
In the first part of Section \ref{secDD} we apply the results of Section \ref{secRep} to
study drawdowns $DD_t = \overline{M}_t - M_t$ and relative drawdowns
$rDD_t = 1-M_t/\overline{M}_t$ of local martingales $(M_t)$.
We provide formulas for the distributions
of $\overline{M}_{T_{\lambda}}$ for stopping times of the form
$T_{\lambda} = \inf \crl{t : M_t = \lambda(\overline{M}_t)}$ and for the random times
$g_{\lambda} = \sup \crl{t \le T_{\lambda} : M_t = \overline{M}_t}$, where
$\lambda$ is a Borel function.
We are also able to calculate the distributions of the maximum drawdown
$\sup_{T \le t \le T_K} DD_t$ and maximum relative drawdown $\sup_{T \le t \le T_K} rDD_t$
for arbitrary stopping times $T$ and
hitting times of the form $T_K = \inf \crl{t \ge T : Y_t = K}$.
In the second part of Section \ref{secDD} we extend these results to
processes $(Y_t)$ which admit a continuous increasing function
$s$ such that $s(Y_t)$ is a local martingale.
In Section \ref{secAppl} we discuss applications
in financial modelling and risk management. First we discuss the
pricing and hedging of options that depend on the running maximum
of an underlying price process $(M_t)$ and are triggered when $(M_t)$
drops to a prespecified level $c \in [0,M_0)$ or when $DD_t$ or $rDD_t$
reach a value $c > 0$. Then we calculate distributions of various random variables
related to diffusions of the form $dY_t = \mu(Y_t) dt + \sigma(Y_t) dB_t$.
An earlier result of Lehoczky \cite{L} appears as a special case.

\setcounter{equation}{0}
\section{Preliminaries}
\label{secPrel}

We start by studying positive and negative parts of processes
of class $\cs$, non-negative processes of class $\cs$ and the
convergence of $X_t$ for $t \to \infty$.

\begin{lemma} \label{lemmasub}
Let $(X_t)$ be a process of class $\cs$. Then
\begin{itemize}
\item[{\rm (1)}] $(X^+_t)$ and $(X^-_t)$ are local submartingales.
\item[{\rm (2)}] If $(X_t)$ has no negative jumps, then
$(X^+_t)$ is again of class $\cs$. If $(X_t)$ has no positive jumps, then
$(X^-_t)$ is of class $\cs$.
\item[{\rm (3)}] If $(X_t)$ is non-negative, then it is a local
submartingale with $A_t = \sup_{u \le t} (-N_u) \vee 0$.
\item[{\rm (4)}] If $(X_t)$ is of class $\csd$, then
$(N_t)$ is a uniformly integrable martingale and $(A_t)$ of
integrable total variation; in particular, there exist integrable
random variables $X_{\infty}, N_{\infty}, A_{\infty}$ such that
$X_t \to X_{\infty}$, $N_t \to N_{\infty}$, $A_t \to A_{\infty}$ almost surely and
in $L^1$.
\end{itemize}
\end{lemma}

\begin{proof}
(1) Since $(A_t)$ is continuous, one has
$\int_0^t 1_{\crl{X_{u-} > 0}} dA_u = \int_0^t 1_{\crl{X_u > 0}} dA_u = 0$.
So Tanaka's formula yields
\be \label{repX+}
X^+_t = X^+_0 + \int_0^t 1_{\crl{X_{u-} > 0}} dX_u + V_t
= X^+_0 + \int_0^t 1_{\crl{X_{u-} > 0}} dN_u + V_t
\ee
for the increasing finite variation process
$$
V_t = \sum_{0 < u \le t} 1_{\crl{X_{u-} \le 0}} X^+_u
+ \sum_{0 < u \le t} 1_{\crl{X_{u-} > 0}} X^-_u + \frac{1}{2} l_t
$$
and $(l_t)$ the local time of $(X_t)$ at $0$ (see, for instance, Protter \cite{Pr}).
This shows that $(X^+_t)$ is a local submartingale. The same is true for
$(X^-_t)$ because $(-X_t)$ is also of class $\cs$.

(2) If $(X_t)$ has no negative jumps, equation \eqref{repX+} reduces to
\be \label{red}
X^+_t = X^+_0 + \int_0^t 1_{\crl{X_{u-} > 0}} dN_u
+ \sum_{0 < u \le t} 1_{\crl{X_{u-} \le 0}} X^+_u + \frac{1}{2} l_t.
\ee
$(\int_0^t 1_{\crl{X_{u-} > 0}} dN_u)$ is a local martingale and the local time $(l_t)$
is continuous and has the property $\int_0^t 1_{\crl{X_u \neq 0}} dl_u = 0$
for all $t \ge 0$. It remains to show that the process
$Y_t = \sum_{0 < u \le t} 1_{\crl{X_{u-} \le 0}} X^+_u$
can be decomposed into the sum of a local martingale
and an adapted continuous increasing process $(C_t)$ with $C_0 = 0$ and
\be \label{Csupp}
\int_0^t 1_{\crl{X^+_u \neq 0}} dC_u = 0 \quad \mbox{for all } t \ge 0.
\ee
Since $(N_t)$ and $(\int_0^t 1_{\crl{X_{u-} > 0}} dN_u)$ are local martingales
and $(A_t)$ is continuous, there exists a sequence of stopping times
$(T_n)_{n \in \mathbb{N}}$ increasing to $\infty$ such that
$$
\E{(X_{T_n})^+} = \E{(N_{T_n} + A_{T_n})^+} < \infty \quad \mbox{and} \quad
\E{\int_0^{T_n} 1_{\crl{X_{u-} > 0}} dN_u} = 0
$$
for all $n \in \mathbb{N}$.
So it follows from \eqref{red} that
$\E{Y_{T_n}} \le \E{(X_{T_n})^+} < \infty$ for all $n \in \mathbb{N}$. Hence, by
Theorem VI.80 of Dellacherie and Meyer \cite{DMB},
there exists a right-continuous increasing predictable process $(C_t)$
starting at $0$ such that $Y_t - C_t$ is a local martingale.
Since $(A_t)$ is continuous, the jumps of $(X_t)$ coincide with those of $(N_t)$.
Due to the local martingale property of $(N_t)$ and the fact that the jumps are positive,
they have to occur at totally inaccessible stopping times. From
Theorem VI.76 of Dellacherie and Meyer \cite{DMB}, one obtains
$\E{\Delta C_T} = \E{\Delta Y_T} = 0$ for every predictable stopping time
$T < \infty$. This shows that $(C_t)$ is continuous. Moreover, there exists
a sequence of stopping times $(R_n)_{n \in \mathbb{N}}$ increasing to $\infty$
such that
$$
\E{\int_0^{t \wedge R_n} 1_{\crl{X^+_{u-} \neq 0}} dC_u} =
\E{\int_0^{t \wedge R_n} 1_{\crl{X^+_{u-} \neq 0}} dY_u} =
\E{\sum_{0 < u \le t \wedge R_n} 1_{\crl{X^+_{u-} \neq 0}} 1_{\crl{X_{u-} \le 0}} X^+_u} = 0
$$
for all $n \in \mathbb{N}$. By monotone convergence one obtains
$$
\E{\int_0^t 1_{\crl{X^+_{u-} \neq 0}} dC_u} =
\E{\int_0^t 1_{\crl{X^+_{u-} \neq 0}} dY_u} =
\E{\sum_{0 < u \le t} 1_{\crl{X^+_{u-} \neq 0}} 1_{\crl{X_{u-} \le 0}} X^+_u} = 0.
$$
This shows \eqref{Csupp} and proves that $(X^+_t)$ is of class $\cs$.
That $(X^-_t)$ is of class $\cs$ if $(X_t)$ has no positive jumps follows
from the same arguments applied to $(-X_t)$.

(3) If $(X_t)$ is non-negative, it follows from (1) that it is a local
submartingale. Hence, $A_t \ge A_u \ge -N_u \vee 0$ for all $t \ge u$, and therefore,
$A_t \ge \sup_{u \le t}(-N_u) \vee 0$. Now assume
\be \label{p>0}
\p[A_t > \sup_{u \le t}(-N_u) \vee 0] > 0
\ee
and introduce the random time
$T = \sup \crl{s \le t : A_s = \sup_{u \le s}(-N_u) \vee 0}$.
Since $(A_s)$ is continuous and $\sup_{u \le s}(-N_u) \vee 0$
increasing, one has $A_T = \sup_{u \le T}(-N_u) \vee 0$.
Moreover, since $X_u > 0$ on the stochastic interval
$\crl{(u,\omega) : T(\omega) < u \le t}$, it follows from
$\int_0^t 1_{\crl{X_u \neq 0}} dA_u = 0$ that $A_t = A_T$, a contradiction
to \eqref{p>0}. Hence, $A_t = \sup_{u \le t}(-N_u) \vee 0$.

(4) If $(X_t)$ is of class $\csd$, then $(X^+_t)$ and $(X^-_t)$ are
submartingales of class $\cd$. Therefore both have a Doob--Meyer
decomposition into the sum of a uniformly integrable martingale and
a predictable increasing process of integrable total variation:
$$
X^+_t = N^1_t + V^1_t \, , \quad X^-_t = N^2_t + V^2_t.
$$
Since the predictable finite variation part of a special semimartingale
is unique, one has $N_t = N^1_t - N^2_t$ and $A_t = V^1_t - V^2_t$.
So $(N_t)$ is a uniformly integrable martingale and $(A_t)$ of
integrable total variation. It follows that there exist integrable
random variables $X_{\infty}, N_{\infty}, A_{\infty}$ such that
$X_t \to X_{\infty}$, $N_t \to N_{\infty}$, $A_t \to A_{\infty}$
almost surely and in $L^1$.
\end{proof}

The next lemma shows that the product of processes of class $\cs$
with vanishing quadratic covariations is again of class $\cs$.

\begin{lemma} \label{lemmaprod}
Let $(X^1_t), \dots, (X^n_t)$ be processes of class $\cs$
such that $\edg{X^i,X^j}_t = 0$ for $i \neq j$.
Then $\prod_{i=1}^n X^i_t$ is again of class $\cs$.
\end{lemma}

\begin{proof}
Since $[X^1,X^2]_t = 0$, integration by parts yields
$$
X^1_t X^2_t = X^1_0 X^2_0 + \int_0^t X^1_{u-} dN^2_u
+ \int_0^t X^2_{u-} dN^1_u + \int_0^t X^1_u dA^2_u
+ \int_0^t X^2_u dA^1_u.
$$
$\int_0^t X^1_{u-} dN^2_u + \int_0^t X^2_{u-} dN^1_u$ is a local martingale
and $\int_0^t X^1_u dA^2_u + \int_0^t X^2_u dA^1_u$ a continuous finite
variation process starting at $0$ which only moves when $X^1_t = 0$ or
$X^2_t = 0$. Hence, $X^1_t X^2_t$ is of class $\cs$. If $n \ge 3$, then
$\edg{X^1 X^2, X^3}_t = 0$, and the lemma follows by induction.
\end{proof}

In the following lemma and the subsequent corollary we
extend results of Nikeghbali \cite{NSub} to our framework that will
be needed later in the paper.

\begin{lemma} \label{lemmatrans}
Let $(X_t)$ be a process of class $\cs$ and $f : \mathbb{R} \to \mathbb{R}$
a locally bounded Borel function. Denote $F(x) = \int_0^x f(y) dy$.
Then the following hold:\\[2mm]
{\rm (1)} The process $f(A_t) X_t$ is again of class $\cs$ with decomposition
\be \label{decomp}
f(A_t)X_t = f(0) X_0 + \int_0^t f(A_u) dN_u + F(A_t).
\ee
{\rm (2)}
If $(f(A_t) X_t)$ is of class $\cd$, then
$M_t = f(A_t)X_t - F(A_t)$ is a uniformly integrable martingale,
and therefore,
\be \label{fAoptst}
f(A_T) X_T - F(A_T) = \E{M_{\infty} \mid {\cal F}_T} \quad
\mbox{for every stopping time } T.
\ee
\end{lemma}

\begin{proof}
(1) It can easily be checked that $f(A_t) X_t$ is c\`adl\`ag. To show that
it is of class $\cs$ with decomposition \eqref{decomp} we first assume
that $f$ is $C^1$. Then
$$
f(A_t) X_t = f(0) X_0 + \int_0^t f(A_u) (dN_u + dA_u)
+ \int_0^t X_u f'(A_u) dA_u.
$$
Since $(X_t)$ is of class $\cs$, the last integral vanishes. So
$$
f(A_t) X_t = f(0) X_0 + \int_0^t f(A_u) dN_u + F(A_t).
$$
$f(0) X_0 + \int_0^t f(A_u) dN_u$ is a local martingale and $F(A_t)$
a continuous adapted finite variation process starting at $0$. Moreover,
$$
\int_0^t 1_{\crl{X_u \neq 0}} dF(A_u) = \int_0^t 1_{\crl{X_u \neq 0}} f(A_u) dA_u = 0,
$$
and therefore also,
$$
\int_0^t 1_{\crl{f(A_u) X_u \neq 0}} dF(A_u) = 0,
$$
which shows that $f(A_t)X_t$ is of class $\cs$. The case where
$f$ is a bounded Borel function now follows from a monotone class argument.
From there it can be extended to locally bounded Borel functions by localization
with a sequence of stopping times.

(2)
If $f(A_t)X_t$ is of class $\csd$, it follows from Lemma \ref{lemmasub} that its
local martingale part $M_t = f(A_t)X_t - F(A_t)$
is a uniformly integrable martingale. Formula \eqref{fAoptst}
is then a consequence of Doob's optional stopping theorem.
\end{proof}

Note that if $(M_t)$ is a local martingale starting at $m \in \mathbb{R}$
such that $(\overline{M}_t)$ is continuous, then $X_t = \overline{M}_t - M_t$
is of class $\cs$ with decomposition $X_t = (m - M_t) + (\overline{M}_t - m)$.
So one obtains from Lemma \ref{lemmatrans} that for every
locally bounded Borel function $f : \mathbb{R} \to \mathbb{R}$
and $F(x) = \int_0^x f(y) dy$, the process
$$
F(A_t) - f(A_t)X_t = F(\overline{M}_t-m) - f(\overline{M}_t-m)(\overline{M}_t - M_t)
$$
is again a local martingale. This transformation was used by
Az\'ema and Yor \cite{AY} in their solution of the Skorokhod embedding problem.
In Carraro et al. \cite{CEO} it is studied for max-continuous semimartingales.

One can use Lemma \ref{lemmatrans} to calculate the probability that
processes of the form $f(A_t)X_t$ stay below a given constant,
which without loss of generality, can be taken to be $1$.
This will prove useful in the study of drawdowns and relative
drawdowns in Section \ref{secDD}.

\begin{corollary} \label{corfiniteR}
Let $(X_t)$ be a non-negative process of class $\cs$ with no positive jumps
such that $A_{\infty} = \infty$, $f : \mathbb{R}_+ \to \mathbb{R}_+$ a Borel function
and $T < \infty$ a stopping time. Then
\be \label{hp1}
\begin{split}
& \p[f(A_t) X_t < 1 \mbox{ for all } t \ge T \mid {\cal F}_T]
= \p[f(A_t) X_t \le 1 \mbox{ for all } t \ge T \mid {\cal F}_T]\\
&= (1 - f(A_T)X_T)^+ \exp \brak{- \int_{A_T}^{\infty} f(x) dx}.
\end{split}
\ee
Moreover, in both of the following two cases:\\[1mm]
{\rm (1)} $K$ is an ${\cal F}_T$-measurable random variable such that
$K > A_T$ and $T_K = \inf \crl{t : A_t \ge K}$\\[1mm]
{\rm (2)} $K$ is an ${\cal F}_T$-measurable random variable such that
$K \ge A_T$ and $T_K = \inf \crl{t : A_t > K}$,\\[2mm]one has
\be \label{hp2}
\begin{split}
& \p[f(A_t) X_t < 1 \mbox{ for all } t \in [T, T_K] \mid {\cal F}_T]
= \p[f(A_t) X_t \le 1 \mbox{ for all } t \in [T, T_K] \mid {\cal F}_T]\\
&= (1 - f(A_T)X_T)^+ \exp \brak{- \int_{A_T}^K f(x) dx}.
\end{split}
\ee
\end{corollary}

\begin{proof}
Let us first assume that
$f$ is bounded and $F(\infty) = \int_0^{\infty} f(y) dy < \infty$.
Then one obtains from part (1) of Lemma \ref{lemmatrans} that
$Y_t = f(A_t) X_t$ is a non-negative process of class $\cs$
with no positive jumps. For a given stopping time $T < \infty$,
denote $R = \inf \crl{t \ge T : Y_t \ge 1}$.
By \eqref{decomp}, $Y_t$ decomposes as
$$
Y_t = f(0)X_0 + \int_0^t f(A_u)dN_u + F(A_t)
$$
and
$$
e^{F(A_t)} Y_t = f(0)X_0 + \int_0^t e^{F(A_u)} f(A_u) dN_u + e^{F(A_t)}.
$$
is again of class $\cs$. In particular, $e^{F(A_t)}(1 - Y_t)$
is a local martingale and
$$
M_t = 1_{\crl{t \ge T}} \brak{e^{F(A_{R \wedge t})}(1-Y_{t \wedge R})
- e^{F(A_T)}(1-Y_T)}
$$
a bounded martingale such that $M_0 = 0$ and $M_t \to M_{\infty}$
almost surely and in $L^1$ for an integrable random variable $M_{\infty}$.
Note that $M_{\infty} = 0$ on $\crl{T = R}$ and $M_{\infty} = - e^{F(A_T)}(1-Y_T)$ on
$\crl{T < R < \infty}$. Moreover, since $A_{\infty} 1_{\crl{L < \infty}}
= A_L 1_{\crl{L < \infty}}$ is real-valued,
it follows from $A_{\infty} = \infty$ that $L = \infty$. Hence, there
exists a sequence $(T_n)_{n \in \mathbb{N}}$ of stopping times that increase to $\infty$
almost surely such that $Y_{T_n} = 0$ for all $n \in \mathbb{N}$, and one
obtains
$$
M_{\infty} = \lim_{n \to \infty}
e^{F(A_{T_n})}(1-Y_{T_n}) - e^{F(A_T)}(1-Y_T)
= e^{F(\infty)} - e^{F(A_T)}(1-Y_T)
$$
almost everywhere on $\crl{R = \infty}$.
So $\E{M_{\infty} \mid {\cal F}_T} = 0$ yields
$$
e^{F(A_T)} (1-Y_T)^+ = \p[R = \infty \mid {\cal F}_T] e^{F(\infty)},
$$
which is equivalent to
\be \label{e1}
\p[f(A_t) X_t < 1 \mbox{ for all } t \ge T \mid {\cal F}_T]
= (1 - f(A_T)X_T)^+ \exp \brak{- \int_{A_T}^{\infty} f(x) dx}.
\ee
The equality
\be \label{e2}
\p[f(A_t) X_t \le 1 \mbox{ for all } t \ge T \mid {\cal F}_T]
= (1 - f(A_T)X_T)^+ \exp \brak{- \int_{A_T}^{\infty} f(x) dx}
\ee
follows from the same argument applied to the
stopping time $\tilde{R} = \inf \crl{t \ge T : Y_t > 1}$.
That \eqref{e1} and \eqref{e2} still hold for
general Borel functions $f : \mathbb{R}_+ \to \mathbb{R}_+$
can be seen by approximating $f$ with $f^n = f \wedge n 1_{[0,n]}$, $n \in \mathbb{N}$.
Note that the functions $f^n$ increase to $f$ and for every $x \ge 0$ there
exists an $n_0 \in \mathbb{N}$ such that $f^n(x) = f(x)$ for all $n \ge n_0$.
Therefore, one has
$$
\bigcap_{n \in \mathbb{N}} \crl{f^n(A_t) X_t < 1 \mbox{ for all } t \ge T}
= \crl{f(A_t) X_t < 1 \mbox{ for all } t \ge T}
$$
as well as
$$
\bigcap_{n \in \mathbb{N}} \crl{f^n(A_t) X_t \le 1 \mbox{ for all } t \ge T}
= \crl{f(A_t) X_t \le 1 \mbox{ for all } t \ge T}.
$$
In case (2) one obtains \eqref{hp2} from \eqref{hp1} simply
by setting $f$ equal to $0$ on $(K, \infty)$. In case (1),
setting $f$ equal to $0$ on $[K, \infty)$ gives
\beas
&& \p[f(A_t) X_t < 1 \mbox{ for all } t \in [T, T_K) \mid {\cal F}_T]
= \p[f(A_t) X_t \le 1 \mbox{ for all } t \in [T, T_K) \mid {\cal F}_T]\\
&& = (1 - f(A_T)X_T)^+ \exp \brak{- \int_{A_T}^K f(x) dx}.
\eeas
But this is equivalent to \eqref{hp2} since $X_{T_K} = 0$.
\end{proof}

\setcounter{equation}{0}
\section{Representation results}
\label{secRep}

\subsection{Representations in terms of last passage times}
\label{subsecRepLP}

The results in this subsection are
inspired by a representation formula for relative martingales by
Az\'ema and Yor \cite{AYzeros} and the recent formula \eqref{MRYform}
of Madan et al. \cite{MRY}. In fact, in the special case
$f \equiv 1$, part (1) of Theorem \ref{thmrepL} below
follows from the proof of Proposition 2.2.a) in Az\'ema and Yor \cite{AYzeros}.
In part (2) of Theorem \ref{thmrepL} and
Corollary \ref{correpL} we relax the integrability conditions on $(X_t)$.
This leads to formulas involving
conditional expectations of random variables which
are conditionally integrable but not necessarily integrable.
To cover this case, we define the conditional expectation of any
random variable $X$ with respect to a sub-$\sigma$-algebra ${\cal G}$
of ${\cal F}$ by
\be \label{extE}
\E{X \mid {\cal G}}
= \sup_{m \in \mathbb{Z}} \inf_{n \in \mathbb{Z}}
\E{m \wedge (n \vee X) \mid {\cal G}}.
\ee
Then
$$
\E{X Y \mid {\cal G}} = X \E{Y \mid {\cal G}}
$$
for all ${\cal G}$-measurable random variables $X$ and
integrable random variables $Y$.

\begin{theorem} \label{thmrepL}
Let $(X_t)$ be a process of class $\cs$ and $f : \mathbb{R} \to \mathbb{R}$
a Borel function. Then the following hold:
\begin{itemize}
\item[{\rm (1)}] If $(X_t)$ is of class $\cd$, then there exist integrable
random variables $X_{\infty}$, $N_{\infty}$, $A_{\infty}$ such that
$X_t \to X_{\infty}$, $N_t \to N_{\infty}$, $A_t \to A_{\infty}$
almost surely as well as in $L^1$ and
\be \label{repL}
f(A_T) X_T = \E{f(A_{\infty}) X_{\infty} 1_{\crl{L \le T}} \mid {\cal F}_T} \quad
\mbox{for every stopping time } T.
\ee
\item[{\rm (2)}]
If $q : \mathbb{R} \to \mathbb{R} \setminus \crl{0}$ is a
Borel function such that $q(A_t) X_t$ is of class $\cd$, then there exist
random variables $X_{\infty}$, $N_{\infty}$, $A_{\infty}$ such that
$X_t \to X_{\infty}$, $N_t \to N_{\infty}$, $A_t \to A_{\infty}$
almost everywhere on $\crl{L < \infty}$ and
\be \label{repTfinite}
f(A_T) X_T = \E{f(A_{\infty}) X_{\infty} 1_{\crl{L \le T}} \mid {\cal F}_T} \quad
\mbox{for all stopping times } T < \infty.
\ee
\end{itemize}
In particular, in both cases one has
\be \label{repTfiniteX}
X_T = \E{X_{\infty} 1_{\crl{L \le T}} \mid {\cal F}_T} \quad
\mbox{for all stopping times } T < \infty.
\ee
\end{theorem}

\begin{proof}
(1) If $(X_t)$ is of class $\csd$, it follows from Lemma \ref{lemmasub} that
$(N_t)$ is a uniformly integrable martingale and $(A_t)$ of integrable
total variation. So there exist integrable random variables
$X_{\infty}$, $N_{\infty}$, $A_{\infty}$
such that $X_t \to X_{\infty}$, $N_t \to N_{\infty}$, $A_t \to A_{\infty}$
almost surely as well as in $L^1$, and for every stopping time $T$, the following
trick from the proof of Proposition 2.2 of Az\'ema and Yor \cite{AYzeros}
can be applied: Denote
$$
d_T = \inf \crl{t > T : X_t = 0} \quad \mbox{with the convention }
\inf \emptyset = \infty.
$$
Since $X_{\infty} 1_{\crl{L \le T}} = X_{d_T}$ and $A_T = A_{d_T}$,
it follows from Doob's optional stopping theorem that
$$
\E{X_{\infty} 1_{\crl{L \le T}} \mid {\cal F}_T}
= \E{N_{d_T} + A_{d_T} \mid {\cal F}_T} = N_T + A_T = X_T.
$$
Moreover, one has $A_{\infty} = A_T$ almost everywhere on $\crl{L \le T}$,
and therefore,
$$
\E{f(A_{\infty}) X_{\infty} 1_{\crl{L \le T}} \mid {\cal F}_T}
= \E{f(A_T) X_{\infty} 1_{\crl{L \le T}} \mid {\cal F}_T} = f(A_T) X_T.
$$
(2) If there exists a Borel function
$q : \mathbb{R} \to \mathbb{R} \setminus \crl{0}$
such that $q(A_t) X_t$ is of class $\cd$, then
$h(x) = |q(x)| \wedge 1$ is a bounded Borel function
and $Y_t = h(A_t) X_t$ is still of class $\cd$. It follows from
Lemma \ref{lemmatrans} that $(Y_t)$ is of class $\csd$. By (1),
$Y_t \to Y_{\infty}$ almost surely as well as in $L^1$ and
$$
Y_T = \E{Y_{\infty} 1_{\crl{L \le T}} \mid {\cal F}_T}
\quad \mbox{for every stopping time } T.
$$
Since $\int_0^t 1_{\crl{X_u \neq 0}} dA_u = 0$ for all
$t \ge 0$, $A_t$ converges to $A_L$ almost everywhere on $\crl{L < \infty}$.
Hence, it follows from $h \neq 0$ that $X_t \to X_{\infty} = Y_{\infty}/h(A_L)$
and $N_t \to N_{\infty} = X_{\infty} - A_L$
almost everywhere on $\crl{L < \infty}$. On $\crl{L=\infty}$, set
$X_{\infty} = N_{\infty} = A_{\infty} = 0$. If $T$ is a stopping time
satisfying $T < \infty$, then
$$
f(A_T) X_T = \frac{f(A_T)}{h(A_T)} Y_T
= \E{\frac{f(A_T)}{h(A_T)} Y_{\infty} 1_{\crl{L \le T}} \mid {\cal F}_T}
= \E{f(A_{\infty}) X_{\infty} 1_{\crl{L \le T}} \mid {\cal F}_T}.
$$
\end{proof}

\begin{corollary} \label{correpL}
Let $(X_t)$ be a process of class $\cs$ and $f :\mathbb{R} \to \mathbb{R}$
a Borel function. Assume that at least one of the following two conditions holds:
\begin{itemize}
\item[{\rm (1)}]
$(N_t)$ is a uniformly integrable martingale
\item[{\rm (2)}]
$(X^-_t)$ and $(N^+_t)$ are of class $\cd$.
\end{itemize}
Then there exist random variables $X_{\infty}$, $N_{\infty}$, $A_{\infty}$ such that
$X_t \to X_{\infty}$, $N_t \to N_{\infty}$, $A_t \to A_{\infty}$ almost everywhere
on $\crl{L < \infty}$ and
$$
f(A_T) X_T = \E{f(A_{\infty}) X_{\infty} 1_{\crl{L \le T}} \mid {\cal F}_T} \quad
\mbox{for all stopping times } T < \infty.
$$
In particular,
$$
X_T = \E{X_{\infty} 1_{\crl{L \le T}} \mid {\cal F}_T} \quad
\mbox{for all stopping times } T < \infty.
$$
\end{corollary}

\begin{proof}
In both cases $e^{-|A_t|} X_t$ is of class $\cd$. So the corollary follows
from part (2) of Theorem \ref{thmrepL}.
\end{proof}

\begin{Remark}
For representations of the form \eqref{repL}, \eqref{repTfinite}
or \eqref{repTfiniteX} to hold
it is not sufficient that a process $(X_t)$ of class $\cs$ has an almost sure finite limit
$\lim_{t \to \infty} X_t$. For example,
$X_t = 1 - \exp(B_t - t/2)$ is of class $\cs$
with $X_0 = 0$ and $\lim_{t \to \infty} X_t = 1$ almost surely. But
$X_t = \p[L \le t \mid {\cal F}_t]$
cannot hold since there is a positive probability that $X_t$ is negative
and $\p[L \le t \mid {\cal F}_t]$ is always between $0$ and $1$.

Processes $(X_t)$ of class $\cs$
that are not of class $\cd$ but satisfy \eqref{repTfinite} and
\eqref{repTfiniteX} can be constructed
from strict local martingales as follows: Take a non-negative continuous
strict local martingale $(M_t)$ starting at
$m \in \mathbb{R}_+ \setminus \crl{0}$ such that
$\lim_{t \to \infty} M_t = 0$ almost surely
(for instance, $M_t = \N{B_t}^{-1}_2$ for a 3-dimensional Brownian motion
starting from a point $x \in \mathbb{R}^3 \setminus \crl{0}$ and
$\N{.}_2$ the Euclidean norm on $\mathbb{R}^3$). $(M_t)$ is a supermartingale but
not a martingale. So there exists
$u \in \mathbb{R}_+$ such that $\E{M_u} < m$, and it follows
from Lemma 2.1 and Proposition 2.3 of Elworthy et al. \cite{ELY}
that $\E{\overline{M}_t} = \infty$ for all $t \ge u$. Hence,
$X_t = \overline{M}_t - M_t$ is a non-negative process of class $\cs$
with $\lim_{t \to \infty} X_t = \overline{M}_{\infty}$ almost surely
and $\E{X_t} = \infty$ for all $t \ge u$.
Clearly, $(X_t)$ satisfies condition (2) of Corollary \ref{correpL}.
So
$$
f(A_T) X_T = \E{f(A_{\infty}) X_{\infty} 1_{\crl{L \le T}} \mid {\cal F}_T}
$$
for every Borel function $f : \mathbb{R} \to \mathbb{R}$ and stopping time
$T < \infty$, even though $X_{\infty}$ is not integrable and the conditional
expectation has to be understood in the sense of \eqref{extE}.
\end{Remark}

As a consequence of Lemma \ref{lemmasub} and Theorem \ref{thmrepL},
one obtains the following

\begin{corollary} \label{corMRY}
{\rm (Madan--Roynette--Yor \cite{MRY})}\\
Let $K$ be a constant and $(M_t)$ a local martingale
with no positive jumps such that $(M^-_t)$ is of class $\cd$.
Denote $g_K = \sup \crl{t : M_t \ge K}$. Then
\be \label{MRY}
(K-M_T)^+ = \E{(K-M_{\infty})^+ 1_{\crl{g_K \le T}} \mid {\cal F}_T},
\ee
for every stopping time $T$. In particular, if $M_{\infty} = m \in \mathbb{R}$, then
$$
(K-M_T)^+ = (K-m)^+ \p[g_K \le T \mid {\cal F}_T].
$$
\end{corollary}

\begin{proof}
$K - M_t$ is a local martingale with no negative jumps. So it follows from
Lemma \ref{lemmasub} that $(K-M_t)^+$ is a local submartingale of class $\cs$.
Since $(M^-_t)$ is of class $\cd$, $(K-M_t)^+$ is of class $\csd$ and
\eqref{MRY} follows from Theorem \ref{thmrepL} by noting that
$g_K = \sup \crl{t : (K-M_t)^+ = 0}$.
\end{proof}

\begin{Remark}
If $K$ is a constant and $(M_t)$ a local martingale with no negative jumps such that
$(M^+_t)$ is of class $\cd$, one can apply Corollary \ref{corMRY} to $-K$, $(-M_t)$
and $g_K = \sup \crl{t : M_t \le K}$. This gives
\be \label{MRYcall}
(M_T - K)^+ = \E{(M_{\infty} - K)^+ 1_{\crl{g_K \le T}} \mid {\cal F}_T}
\ee
for all stopping times $T$. In particular, if
$M_{\infty} = m \in \mathbb{R}$, then
\be \label{MRYcall1}
(M_T - K)^+ = (m-K)^+ \p[g_K \le T \mid {\cal F}_T].
\ee
However, if for instance, $M_t = \exp(B_t - t/2)$ for a Brownian motion $(B_t)$,
the assumptions of Corollary \ref{corMRY} are satisfied but
$(M^+_t)$ is not of class $\cd$. So even though $M_{\infty} = 0$,
formula \eqref{MRYcall1} does not hold. Indeed, for $K \ge 0$ and
$T = t \in \mathbb{R}_+ \setminus \crl{0}$, the right-hand side is
zero but $\p[(M_t - K)^+ > 0] > 0$. For a more detailed discussion of this case,
we refer to Section 6 in Madan et al. \cite{MRY}.
\end{Remark}

The following extension of Corollary \ref{corMRY} has been proved
by Profeta et al. \cite{PRY} with methods from the
theory of enlargement of filtrations. We can deduce it under slightly
weaker assumptions from Lemma \ref{lemmaprod} and Theorem \ref{thmrepL}.

\begin{corollary} \label{corPRY}
{\rm (Profeta--Roynette--Yor \cite{PRY})}\\
Let $K^1, \dots, K^n$ be constants and $(M^1_t), \dots, (M^n_t)$
local martingales that are bounded from below and have no positive jumps.
Assume $[M^i,M^j]_t = 0$ for $i \neq j$ and denote
$g^i = \sup \crl{t : M^i_t \ge K^i}$. Then
\be \label{py}
\prod_{i=1}^n (K^i- M^i_T)^+ =
\E{\prod_{i=1}^n (K^i- M^i_{\infty})^+
1_{\crl{g^i \le T}} \mid {\cal F}_T},
\ee
for every stopping time $T$. In particular, if $M^i_{\infty} = m^i \in \mathbb{R}$
for all $i = 1, \dots, n$, then
$$
\prod_{i=1}^n (K^i- M^i_T)^+ =
\prod_{i=1}^n (K^i-m^i)^+ \p \edg{\bigvee_{i=1}^n g^i \le T \mid {\cal F}_T}
$$
\end{corollary}

\begin{proof}
By Lemma \ref{lemmasub}, $X^i_t = (K^i - M^i_t)^+$ are local submartingales of class $\cs$
such that $\edg{X^i,X^j}_t = 0$ for $i \neq j$. So we obtain from Lemma \ref{lemmaprod}
that $\prod_{i=1}^n X^i_t$ is again of class $\cs$, which since
all $(M^i_t)$ are bounded from below, is bounded. Now \eqref{py} follows from
Theorem \ref{thmrepL}.
\end{proof}

\begin{Remark}
If $(X_t)$ satisfies the assumptions of part (2) of Theorem \ref{thmrepL} or
Corollary \ref{correpL}, then there exists a random variable $X_{\infty}$ such that
$X_t \to X_{\infty}$ almost everywhere on the set $\crl{L < \infty}$,
and one has
\be \label{Xrep}
X_t = \E{X_{\infty} 1_{\crl{L \le t}} \mid {\cal F}_t} \quad
\mbox{for all } t \ge 0.
\ee
In particular, the whole process $(X_t)$ can be recovered from $X_{\infty}$
and $L$. If $(X_t)$ is non-negative, equation \eqref{Xrep} can be rewritten as
\be \label{QMRY}
\E{1_{F_t} X_t} = {\cal Q}[F_t \cap \crl{g \le t}] \quad
\mbox{for every } t \ge 0 \mbox{ and all } F_t \in {\cal F}_t,
\ee
where $g=L$ and ${\cal Q}$ is the $\sigma$-finite measure given by
$d{\cal Q}/d\p = X_{\infty}$. Madan et al. \cite{MRY} raised the question
for which non-negative submartingales $(X_t)$ is it possible to find
a random time $g$ and a $\sigma$-finite measure ${\cal Q}$
such that \eqref{QMRY} holds.
It turns out that if $(X_t)$ satisfies \eqref{QMRY}, where $g$ is the end
of an optional set with the property
\be \label{avoid}
\p[g = T] = 0 \quad \mbox{for all stopping times } T
\ee
and ${\cal Q}$ is of the form $d {\cal Q}/d\p = X$ for an
integrable random variable $X > 0$,
then $(X_t)$ is a submartingale of class $\csd$ with
$X_{\infty} = \E{X \mid \bigvee_t{\cal F}_t}$ and $L = g$.
For $X = 1$, this holds because in this case, $(X_t)$ can be written
as $X_t = 1 - Z^g_t$ for the Az\'ema supermartingale
$Z^g_t = \p[g > t \mid {\cal F}_t]$. Moreover, it follows from \eqref{avoid} that
$g$ is the end of a predictable set and the finite variation part
$A^g_t$ of $Z^g_t$ is continuous. It is shown in Az\'ema \cite{A} that under
these circumstances, one has
$\int_0^t 1_{\crl{Z^g_u \neq 1}} dA^g_u = 0$
and $g = \sup \crl{t : Z^R_t = 1}$, which means that $(X_t)$
is a submartingale of class $\csd$ with $L = g$.
The case of general integrable $X > 0$ follows from an application
of the optional section theorem (see page 136 of Dellacherie et al. \cite{DMM}).
A more thorough discussion of this
problem, also treating the case where ${\cal Q}$ is a $\sigma$-finite measure,
is the subject of the paper Najnudel and Nikeghbali \cite{NN}.
\end{Remark}

\subsection{Stochastic integral representations and conditional distributions}
\label{subsecStoInt}

We here use Lemma \ref{lemmatrans} to derive representation results for non-negative
processes $(X_t)$ of class $\cs$ and Borel functions $f : \mathbb{R}_+ \to \mathbb{R}_+$
such that $f(A_t)X_t$ converges to $1$ as $t \to \infty$. In Section 4 they will
be applied in situations where $f(A_t)X_t$ can be stopped with a stopping time
$R$ such that $f(A_R)X_R = 1$.

\begin{theorem} \label{thmfX}
Let $(X_t)$ be a non-negative process of class $\cs$ and
$f : \mathbb{R}_+ \to \mathbb{R}_+$ a locally bounded Borel function
such that the process $f(A_t)X_t$ is of class $\cd$ and $f(A_t)X_t \to 1$
almost surely. Denote $F(x) = \int_0^x f(y) dy$.

\begin{itemize}
\item[{\rm (1)}]
If $F(\infty) < \infty$, then $A_t = 0$ and $f(0)X_t =1$ for all $t \ge 0$.
\item[{\rm (2)}]
If $F(\infty) = \infty$, then $L < \infty$, $A_L = A_{\infty} < \infty$ and
$X_t \to X_{\infty}$ almost surely for a random variable $X_{\infty} > 0$.
Moreover, for every stopping time $T$ one has
\be \label{fAX}
f(A_T)X_T = \p[L \le T \mid {\cal F}_T]
\ee
and for all Borel functions $h : \mathbb{R}_+ \to \mathbb{R}$ satisfying
\be \label{condh}
\int_0^{\infty} |h(y)| e^{-F(y)} dF(y) < \infty,
\ee
\bea
\label{hedgingA}
\E{h(A_{\infty}) \mid {\cal F}_T}
&=& h(0)f(0)X_0 + h^F(0)(1-f(0)X_0) + \int_0^T (h - h^F)(A_u) f(A_u) dN_u\\
\label{hA}
&=& h(A_T) f(A_T) X_T + h^F(A_T)(1 - f(A_T)X_T),
\eea
where
$$
h^F(x) = e^{F(x)} \int_x^{\infty} h(y) e^{-F(y)} dF(y), \quad x \ge 0.
$$
In particular, conditioned on ${\cal F}_T$, the law of $A_{\infty}$ is given by
\be \label{distA}
\p[A_{\infty} > x \mid {\cal F}_T] = 1_{\crl{A_T > x}}
+ 1_{\crl{A_T \le x}} (1-f(A_T)X_T) e^{F(A_T) - F(x)}, \quad x \ge 0.
\ee
\end{itemize}
\end{theorem}

\begin{proof}
Since $f(A_t)X_t \to 1$ almost surely,
one has $L < \infty$ and $A_t = A_{t \wedge L}$. In particular, $A_L = A_{\infty}$
and there exists a random variable $X_{\infty} > 0$ such that
$X_t \to X_{\infty}$ almost surely.
By Lemma \ref{lemmatrans}, the process $f(A_t)X_t$ is of class
$\csd$. Hence it follows from Theorem \ref{thmrepL} that for every
stopping time $T$,
\be \label{repf}
f(A_T)X_T = \p[L^f \le T \mid {\cal F}_T],
\ee
where $L^f = \sup \crl{t : f(A_t)X_t = 0}$. However, since
$f(A_t)X_t \to 1$, one has $L = L^f$ and \eqref{fAX} follows from \eqref{repf}.
Now let $h : \mathbb{R}_+ \to \mathbb{R}$ be a bounded Borel function.
Then the function
$$
h^F(x) = e^{F(x)} \int_x^{\infty} h(y) e^{-F(y)} dF(y)
$$
is bounded as well. Note that $\Phi(x) = h^F(0) - h^F(x)$ is of the form
$\Phi(x) = \int_0^x \varphi(y) dy$ for $\varphi = (h - h^F) f$.
So one obtains from Lemma \ref{lemmatrans} that $\varphi(A_t) X_t$
is of class $\csd$ and
$$
\varphi(0)X_0 + \int_0^T \varphi(A_u) dN_u
= \varphi(A_t)X_T - \Phi(A_T)
= \E{\varphi(A_{\infty})X_{\infty} - \Phi(A_{\infty}) \mid {\cal F}_T}
$$
for every stopping time $T$. This yields
\bea
\label{hfX1}
\E{h(A_{\infty}) \mid {\cal F}_T}
&=& h(0) f(0) X_0 + h^F(0)(1-f(0)X_0) + \int_0^T (h-h^F)(A_u) f(A_u) dN_u\\
\label{hfX2}
&=& h(A_T) f(A_T) X_T + h^F(A_T)(1 - f(A_T)X_T).
\eea
\eqref{hfX2} applied to $h \equiv 1$ gives
$e^{F(A_t) - F(\infty)} (1 - f(A_t)X_t) = 0$ for all $t \ge 0$.
Hence, for $F(\infty) < \infty$, one must have $f(A_t)X_t = 1$ for all $t \ge 0$, and (1) follows.
If $F(\infty) = \infty$, then formula \eqref{hfX2}
is equivalent to \eqref{distA}, which shows that
condition \eqref{condh} implies $\E{|h(A_{\infty})|\mid {\cal F}_T} < \infty$.
So both equations \eqref{hfX1} and \eqref{hfX2} extend from
bounded $h$ to functions that satisfy \eqref{condh}.
\end{proof}

Theorem \ref{thmfX} will allow us to obtain general results on drawdown and
relative drawdown processes of local martingales in Section \ref{secDD}.
But to extend these results to diffusions we will need the
following generalization of Theorem \ref{thmfX}. The proof is
similar but involves an additional approximation argument.

\begin{theorem} \label{thmfXg}
Let $(X_t)$ be a non-negative process of class $\cs$ and
$f : \mathbb{R}_+ \to \mathbb{R}_+$ a Borel function for which
there exists an increasing sequence $(a_n)_{n \in \mathbb{N}}$
in $(0,\infty)$ such that $f 1_{[0,a_n]}$ is bounded for all $n \in \mathbb{N}$, and
$f(x) = 0$ for $x \ge a = \lim_{n \to \infty} a_n$. Denote
$F(x) = \int_0^x f(y) dy$ and assume
that the process $f(A_t)X_t$ is of class $\cd$ and $f(A_t)X_t \to 1$
almost surely.

\begin{itemize}
\item[{\rm (1)}]
If $F(a) < \infty$, then $A_t = 0$ and $f(0)X_t =1$ for all $t \ge 0$.
\item[{\rm (2)}]
If $F(a) = \infty$, then $L < \infty$, $A_L = A_{\infty} < a$ and
$X_t \to X_{\infty}$ almost surely for a random variable $X_{\infty} > 0$.
Moreover, for every stopping time $T$ one has
\be \label{fAXg}
f(A_T)X_T = \p[L \le T \mid {\cal F}_T]
\ee
and for all Borel functions $h : [0,a) \to \mathbb{R}$ satisfying
\be \label{condhg}
\int_0^a |h(y)| e^{-F(y)} dF(y) < \infty,
\ee
\bea
\label{hedgingAg}
\E{h(A_{\infty}) \mid {\cal F}_T}
&=& h(0)f(0)X_0 + h^F(0)(1-f(0)X_0) + \int_0^T (h - h^F)(A_u) f(A_u) dN_u\\
\label{hAg}
&=& h(A_T) f(A_T) X_T + h^F(A_T)(1 - f(A_T)X_T),
\eea
where
$$
h^F(x) = e^{F(x)} \int_x^a h(y) e^{-F(y)} dF(y), \quad 0 \le x < a.
$$
In particular, conditioned on ${\cal F}_T$, the law of $A_{\infty}$ is
given by
\be \label{distAg}
\p[A_{\infty} > x \mid {\cal F}_T] = 1_{\crl{A_T > x}}
+ 1_{\crl{A_T \le x}} (1-f(A_T)X_T) e^{F(A_T) - F(x)}, \quad x \ge 0.
\ee
\end{itemize}
\end{theorem}

\begin{proof}
It follows as in the proof of Theorem \ref{thmfX} that $L < \infty$,
$A_L = A_{\infty}$ and $X_t \to X_{\infty}$ almost surely for a
random variable $X_{\infty} > 0$.
Now set $f^n = f \wedge n$ and $F^n(x) = \int_0^x f^n(y) dy$, $n \in \mathbb{N}$.
It follows from Lemma \ref{lemmatrans} that the processes
$f^n(A_t)X_t$, $n \in \mathbb{N}$, are of class
$\csd$. So one obtains from Theorem \ref{thmrepL} that
\be \label{repfn}
f^n(A_T)X_T = \E{f^n(A_{\infty})X_{\infty} 1_{\crl{L^n \le T}} \mid {\cal F}_T}
\quad \mbox{for every stopping time } T,
\ee
where $L^n = \sup \crl{t : f^n(A_t)X_t = 0}$. The fact that
$f^n(A_t)X_t \to f^n(A_{\infty}) X_{\infty} > 0$ entails that $L = L^n$, and
one obtains \eqref{fAXg} from \eqref{repfn} by letting $n$ tend to $\infty$.
Now let $h : \mathbb{R}_+ \to \mathbb{R}$ be a bounded Borel function.
Then the functions
$$
h^F(x) = e^{F(x)} \int_x^{\infty} h(y) e^{-F(y)} dF(y) \quad \mbox{and} \quad
h^n(x) = e^{F^n(x)} \int_x^{\infty} h(y) e^{-F^n(y)} dF^n(y), \quad n \in \mathbb{N},
$$
are bounded too. $\Phi^n(x) = h^n(0) - h^n(x)$ can be written as
$\Phi^n(x) = \int_0^t \varphi^n(y) dy$ for $\varphi^n = (h - h^n) f^n$.
So it follow from Lemma \ref{lemmatrans} that $\varphi^n(A_t) X_t$
is of class $\csd$ and
$$
\varphi^n(0)X_0 + \int_0^T \varphi^n(A_u) dN_u
= \varphi^n(A_t)X_T - \Phi^n(A_T)
= \E{\varphi^n(A_{\infty})X_{\infty} - \Phi^n(A_{\infty}) \mid {\cal F}_T}
$$
for every stopping time $T$. In the limit $n \to \infty$, this gives
\bea
\label{hfX1g}
\E{h(A_{\infty}) \mid {\cal F}_T}
&=& h(0) f(0) X_0 + h^F(0)(1-f(0)X_0) + \int_0^T (h-h^F)(A_u) f(A_u) dN_u\\
\label{hfX2g}
&=& h(A_T) f(A_T) X_T + h^F(A_T)(1 - f(A_T)X_T).
\eea
For $h = 1_{[a,\infty)}$ and $T= 0$, the equality between the first and third term
reduces to $\p[A_{\infty} \ge a] = 0$. So $A_{\infty} < a$ and
\eqref{hfX2g} applied to $h \equiv 1$ gives
$e^{F(A_t) - F(a)} (1 - f(A_t)X_t) = 0$ for all $t \ge 0$.
Hence, for $F(a) < \infty$, one must have $f(A_t)X_t = 1$ for all $t \ge 0$, and (1) follows.
If $F(a) = \infty$, formula \eqref{hfX2g}
is equivalent to \eqref{distAg}. This shows that
condition \eqref{condhg} implies $\E{|h(A_{\infty})|\mid {\cal F}_T} < \infty$,
and both equations \eqref{hfX1g} and \eqref{hfX2g} extend from
bounded $h$ to functions satisfying \eqref{condhg}.
\end{proof}

\setcounter{equation}{0}
\section{Drawdown and relative drawdown}
\label{secDD}

\subsection{The local martingale case}
\label{subsecddm}

We first consider a local martingale $(M_t)$ starting
at $m \in \mathbb{R}$ such that the running supremum $(\overline{M}_t)$ is continuous.
Then the drawdown process $DD_t = \overline{M}_t - M_t$ is a non-negative local
submartingale of class $\cs$ with decomposition $(m-M_t) + (\overline{M}_t-m)$.
Moreover, if $m > 0$, then the relative drawdown process
$rDD_t = DD_t/\overline{M}_t$ is well-defined, and by Lemma \ref{lemmatrans},
also a non-negative local submartingale of class $\cs$.
As a consequence of the results of Subsection \ref{subsecRepLP} one obtains
the following

\begin{proposition}
Assume that $(M^-_t)$ is of class $\cd$ and denote
$g = \sup \crl{t : M_t = \overline{M}_t}$. Then there exists a
random variable $DD_{\infty}$ such that
$DD_t \to DD_{\infty}$ almost everywhere on $\crl{g < \infty}$ and
\be \label{DDrepL}
DD_T = \E{DD_{\infty} 1_{\crl{g \le T}} \mid {\cal F}_T}
\ee
for every stopping time $T < \infty$. Moreover, if $m > 0$, then
there exists an integrable random variable $rDD_{\infty}$ such that
$rDD_t \to rDD_{\infty}$ almost surely as well as in $L^1$ and
\bea
\label{rDDrepL}
rDD_T &=& \E{rDD_{\infty}
1_{\crl{g \le T}} \mid {\cal F}_T}\\
\label{rDDhedging}
&=& - \int_0^T \frac{dM_u}{\overline{M}_u} + \log(\overline{M}_t) - \log(m)\\
\label{rDDdist}
&=& \E{rDD_{\infty}
- \log(\overline{M}_{\infty}) \mid {\cal F}_T} + \log(\overline{M}_t).
\eea
for all stopping times $T$.
\end{proposition}

\begin{proof}
$DD_t$ is a process of class $\cs$ with $L = g$ that satisfies the assumptions of
Corollary \ref{correpL}. This shows \eqref{DDrepL}. Moreover, if $m > 0$, then
$rDD_t$ is a process of class $\csd$, and \eqref{rDDrepL}--\eqref{rDDdist} follow
from Theorem \ref{thmrepL} and Lemma \ref{lemmatrans}.
\end{proof}

In the following we study the situation where $(M_t)$ is stopped with a
stopping time of the form
$T_{\lambda} = \inf \crl{t : M_t \le \lambda(\overline{M}_t)}$ for a
Borel function $\lambda : [m, \infty) \to \mathbb{R}$
satisfying $\lambda(x) < x$ for all $x \ge m$.
Denote
\be \label{La}
g_{\lambda} = \sup \crl{t \le T_{\lambda} : M_t = \overline{M}_t},
\quad \Lambda(x) = \int_m^x \frac{dy}{y - \lambda(y)}
\ee
and notice that $\Lambda$ is a well-defined increasing function from
$[m,\infty]$ to $[0,\infty]$.

\begin{proposition} \label{proplambdastop}
Assume the function $1/(x - \lambda(x))$
is locally bounded on $[m,\infty)$ and
\be \label{condst}
\overline{M}_{T_{\lambda}} < \infty \quad \mbox{and} \quad
M_{T_{\lambda}} = \lambda(\overline{M}_{T_{\lambda}}).
\ee
Then $\Lambda(\infty) = \infty$, $g_{\lambda} < T_{\lambda}$,
$M_{T_{\lambda}} < \overline{M}_{g_{\lambda}} = \overline{M}_{T_{\lambda}} $ and
\be \label{repLDD}
\p[g_{\lambda} \le T \mid {\cal F}_T] =
\frac{\overline{M}_T - M_T}{\overline{M}_T
- \lambda(\overline{M}_T)}
\ee
for every stopping time $T \le T_{\lambda}$. Moreover, for all Borel functions
$h : [m,\infty) \to \mathbb{R}$ satisfying
$$
\int_m^{\infty} |h(y)| e^{- \Lambda(y)} d \Lambda(y) < \infty
$$
one has
\bea
\label{hedgingDD}
\E{h(\overline{M}_{T_{\lambda}}) \mid {\cal F}_T}
&=& h^{\Lambda}(m)
+ \int_0^T \frac{h^{\Lambda}(\overline{M}_u)
- h(\overline{M}_u)}{\overline{M}_u - \lambda(\overline{M_u})} dM_u\\
\label{hDD}
&=& \frac{h(\overline{M}_T)(\overline{M}_T - M_T) + h^{\Lambda}(\overline{M}_T)(M_T
- \lambda(\overline{M}_T))}{\overline{M}_T - \lambda(\overline{M_T})},
\eea
where
\be \label{hLa}
h^{\Lambda}(x) = e^{\Lambda(x)} \int_x^{\infty} h(y) e^{- \Lambda(y)} d\Lambda(y),
\quad x \ge m.
\ee
In particular,
\be \label{distbarM}
\p[\overline{M}_{T_{\lambda}} > x \mid {\cal F}_T]
= 1_{\crl{\overline{M}_T > x}} + 1_{\crl{\overline{M}_T \le x}}
\frac{M_T - \lambda(\overline{M}_T)}{\overline{M}_T
- \lambda(\overline{M_T})}
e^{\Lambda(\overline{M}_T) - \Lambda(x)} \quad \mbox{for } x \ge m.
\ee
\end{proposition}

\begin{proof}
$X_t = \overline{M}_{t \wedge T_{\lambda}} - M_{t \wedge T_{\lambda}}$ is a non-negative
process of class $\cs$ starting at $0$ with
decomposition $(m - M_{t \wedge T_{\lambda}}) +
(\overline{M}_{t \wedge T_{\lambda}} - m)$, and
$$
f(x) = \frac{1}{x + m - \lambda(x+m)}
$$
is a non-negative locally bounded Borel function on $\mathbb{R}_+$ with
$F(x) = \int_0^x f(y) dy = \Lambda(x+m)$. By Lemma \ref{lemmatrans}, the process
$$
f(\overline{M}_{t \wedge T_{\lambda}} - m) X_t
= \frac{\overline{M}_{t \wedge T_{\lambda}}
- M_{t \wedge T_{\lambda}}}{\overline{M}_{t \wedge T_{\lambda}}
- \lambda(\overline{M}_{t \wedge T_{\lambda}})}
$$
is of class $\cs$. Moreover, it follows from condition \eqref{condst} that it
takes values in $[0,1]$ and converges to $1$ almost surely.
Since $X_0 = 0$, one obtains from Theorem \ref{thmfX} that
$\Lambda(\infty) = \int_0^{\infty} f(x) dx = \infty$,
$g_{\lambda} < T_{\lambda}$, $M_{T_{\lambda}} < \overline{M}_{g_{\lambda}}
= \overline{M}_{T_{\lambda}}$
and
$$
\p[g_{\lambda} \le T \mid {\cal F}_T] =
\frac{\overline{M}_T - M_T}{\overline{M}_T
- \lambda(\overline{M}_T)}
$$
for all stopping times $T \le T_{\lambda}$. Formulas \eqref{hedgingDD}--\eqref{distbarM}
follow from Theorem \ref{thmfX} applied to the function
$$
\tilde{h}(x) = h (x+m), \quad x \ge 0.
$$
\end{proof}

\begin{Remark} \label{remDoobmax}
If $(M_t)$ is a non-negative local martingale starting at $1$ such that
$(\overline{M}_t)$ is continuous and $M_t \to 0$ almost surely, then formula
\eqref{distbarM} with $T=0$ and $\lambda \equiv 0$ yields that
$1/ \overline{M}_{\infty}$ is uniformly distributed on
the interval $(0,1)$. This is Doob's maximal identity, which has been studied in depth by
Mansuy and Yor \cite{MY} and Nikeghbali and Yor \cite{NY}.
\end{Remark}

The next proposition gives sufficient conditions for assumption
\eqref{condst} to hold.

\begin{proposition} \label{propcond}
Assume $(M_t)$ is continuous.
Then both of the following two conditions imply condition \eqref{condst}.
\begin{itemize}
\item[{\rm (1)}]
$m > 0$, $M_t \to 0$ almost surely and $\lambda(x) \ge 0$ for all $x \ge m$
\item[{\rm (2)}]
$\overline{M}_{\infty} = \Lambda(\infty) = \infty$
\end{itemize}
\end{proposition}

\begin{proof}
Under condition (1) one has $\overline{M}_{T_{\lambda}} \le
\overline{M}_{\infty} < \infty$ and
$$
\frac{\overline{M}_t - M_t}{\overline{M}_t - \lambda (\overline{M}_t)}
\ge \frac{\overline{M}_t - M_t}{\overline{M}_t} \to 1 \mbox{ almost surely,}
$$
which shows that $M_{T_{\lambda}} = \lambda(\overline{M}_{T_{\lambda}})$.

If condition (2) holds, then $\int_0^{\infty} f(x) dx = \infty$ for the function
$$
f(x) = \frac{1}{x + m - \lambda(x + m)}.
$$
Since $X_t = \overline{M}_t - M_t$ is
a non-negative continuous process of class $\cs$ with decomposition
$(m - M_t) + (\overline{M}_t - m)$,
it follows from Corollary \ref{corfiniteR} that
$\p[f(\overline{M}_t-m) X_t <1 \mbox{ for all t }] = 0$. This implies
$T_{\lambda} < \infty$, $\overline{M}_{T_{\lambda}} < \infty$ and
$M_{T_{\lambda}} = \lambda(\overline{M}_{T_{\lambda}})$.
\end{proof}

Several authors have studied the distribution of the maximum drawdown
$\sup_{0 \le t \le T} DD_t$ in the case where $T$ is a constant and
$(M_t)$ a Brownian motion (with or without drift); see, for instance,
Berger and Whitt \cite{BW}, Douady et al. \cite{DSY}, Graversen and Shiryaev \cite{GS},
Magdon-Ismail et al. \cite{MAPA}. With the methods developed here we can
derive conditional distributions of maximum
drawdowns $\sup_{T \le t \le R} DD_t$ and maximum relative drawdowns
$\sup_{T \le t \le R} rDD_t$ when $T \le R$ are suitable stopping times and
$(M_t)$ is a continuous local martingale. In the next subsection we will
generalize these results to the diffusion case.

\begin{proposition} \label{propmdd1}
Assume $(M_t)$ is continuous with $m > 0$ and $M_t \to 0$ almost surely.
Let $T < \infty$ be a stopping time and $K$ an
${\cal F}_T$-measurable random variable such that $0 \le K < M_T$.
Denote $T_K = \inf \crl{t \ge T : M_t = K}$. Then one has for all $x \ge 0$,
\be \label{mdd1}
\begin{split}
\p \edg{\sup_{t \in [T, T_K] \cap \mathbb{R}_+} DD_t > x \mid {\cal F}_T}
= 1_{\crl{\overline{M}_T - K > x}} + 1_{\crl{\overline{M}_T - K \le x}}
\frac{M_T - K}{x}.
\end{split}
\ee
If in addition, $m > 0$, then
\be \label{mrdd1}
\begin{split}
\p \edg{\sup_{t \in [T,T_K] \cap \mathbb{R}_+} rDD_t > x \mid {\cal F}_T}
= 1_{\crl{1 - K/\overline{M}_T > x}} +
1_{\crl{1 - K/\overline{M}_T \le x < 1}} \brak{\frac{M_T - K}{K}} \brak{\frac{1-x}{x}}.
\end{split}
\ee
\end{proposition}

\begin{proof}
Let us first assume $T = 0$ and $K$ is a constant.
Then, by Proposition \ref{propcond}, $(M_t)$ with $\lambda \equiv K$ fulfills the
assumptions of Proposition \ref{proplambdastop}. Moreover,
$$
\sup_{t \in [T,T_K] \cap \mathbb{R}_+} DD_t = \overline{M}_{T_K} -K \quad \mbox{and} \quad
\sup_{t \in [T,T_K] \cap \mathbb{R}_+} rDD_t = 1 - K/\overline{M}_{T_K}.
$$
So \eqref{mdd1} and \eqref{mrdd1} can be deduced from formula \eqref{distbarM}.
In the general case, the proposition follows
by considering the process $\tilde{M}_t = M_{T + t}$ in the filtration
$\tilde{\cal F}_t = {\cal F}_{T+t}$ and conditioning on ${\cal F}_T$.
\end{proof}

\begin{proposition} \label{propmdd2}
Assume $(M_t)$ is continuous and $\overline{M}_{\infty} = \infty$.
Let $T < \infty$ be a stopping time and $K$ a $[m,\infty]$-valued ${\cal F}_T$-measurable
random variable such that $\overline{M}_T < K \le \overline{M}_{\infty}$.
Denote $$T_K = \inf \crl{t \ge T : M_t = K}$$. Then one has for all $x \ge 0$,
\be \label{probMl}
\p[M_t \ge \lambda(\overline{M}_t)
\mbox{ for all } t \in [T, T_K] \cap \mathbb{R}_+ \mid {\cal F}_T]
= \brak{\frac{M_T - \lambda(\overline{M}_T)}{\overline{M}_T
- \lambda(\overline{M}_T)}}^+ e^{\Lambda(\overline{M}_T) - \Lambda(K)}.
\ee
\end{proposition}

\begin{proof}
$X_t = \overline{M}_t -M_t$ is a non-negative process of class $\cs$
with decomposition $(m-M_t) + (\overline{M}_t - m)$ and
$$
f(x) = \frac{1}{x + m - \lambda(x+m)}
$$
a Borel function from $\mathbb{R}_+$ to $\mathbb{R}_+$.
Note that $M_t \ge \lambda(\overline{M}_t)$ is equivalent to
$f(\overline{M_t} - m)X_t \le 1$ and
$T_k = \inf \crl{\overline{M}_t - m = K - m}$. Therefore, it follows
from Corollary \ref{corfiniteR} that
\beas
\p[M_t \ge \lambda(\overline{M}_t)
\mbox{ for all } t \in [T, T_K] \cap \mathbb{R}_+ \mid {\cal F}_T]
&=& \brak{\frac{M_T - \lambda(\overline{M_T})}{M_T
- \lambda(\overline{M}_T)}}^+
\exp \brak{- \int_{\overline{M}_T - m}^{K - m} f(x) dx}\\
&=& \brak{\frac{M_T - \lambda(\overline{M}_T)}{\overline{M}_T
- \lambda(\overline{M}_T)}}^+ \exp \brak{\Lambda(\overline{M}_T) - \Lambda(K)}.
\eeas
\end{proof}

\begin{corollary} \label{cormdd2}
Under the assumptions of Proposition \ref{propmdd2} one has for all $x \ge 0$,
\be \label{mdd2}
\p \edg{\sup_{t \in [T,T_K] \cap \mathbb{R}_+} DD_t \le x \mid {\cal F}_T}
= 1_{\crl{x >0}} \brak{1 - \frac{DD_T}{x}}^+
\exp \brak{\frac{\overline{M}_T - K}{x}},
\ee
and, provided that $m > 0$,
\be \label{mrdd2}
\p \edg{\sup_{t \in [T,T_K] \cap \mathbb{R}_+} rDD_t \le x \mid {\cal F}_T}
= 1_{\crl{x >0}} \brak{1 - \frac{rDD_T}{x}}^+ \brak{\frac{\overline{M}_T}{K}}^{1/x}.
\ee
\end{corollary}

\begin{proof}
First assume $x > 0$. Then formula \eqref{mdd2} follows from
Proposition \ref{propmdd2} applied to the function $\lambda(y) = y-x$ and
\eqref{mrdd2} is obtained by applying Proposition \ref{propmdd2}
with $\lambda(y) = (1-x)y$. Since $\overline{M}_T < K$, one has
$$
\exp \brak{\frac{\overline{M}_T - K}{x}} \to 0 \quad \mbox{and} \quad
\brak{\frac{\overline{M}_T}{K}}^{1/x} \to 0 \quad
\mbox{almost surely for } x \downarrow 0.
$$
This shows that
$$
\p \edg{\sup_{t \in [T,T_K] \cap \mathbb{R}_+} DD_t = 0 \mid {\cal F}_T}
= \p \edg{\sup_{t \in [T,T_K] \cap \mathbb{R}_+} rDD_t = 0 \mid {\cal F}_T} = 0.
$$
\end{proof}

\subsection{The diffusion case}
\label{subsecdiff}

The results of Subsection \ref{subsecddm} can be extended to
stochastic processes that can be turned into local martingales through
a strictly increasing continuous transformation. To do that we here
consider a stochastic process $(Y_t)$ taking values in an interval
$I \subset \mathbb{R}$ which starts at a constant
$y_0 \in I$ such that the supremum process $(\overline{Y}_t)$ is continuous and
there exists a strictly increasing continuous function $s : I \to \mathbb{R}$
making $s(Y_t)$ a local martingale. Our main example is a diffusion of the form
\be \label{SDEY}
dY_t = \mu(Y_t) dt + \sigma(Y_t) dB_t, \quad Y_0 = y_0 \in I,
\ee
where $(B_t)$ is a Brownian motion and $\mu, \sigma: I \to \mathbb{R}$ are
deterministic functions such that
$$
\gamma(x) = 2 \int_{y_0}^x \frac{\mu(y)}{\sigma^2(y)} dy \quad \mbox{and}
\quad \int_{y_0}^x e^{- \gamma(y)} dy \quad
\mbox{are finite for all } x \in I.
$$
Then $s$ can be chosen as $s(x) = c + d \int_{y_0}^x e^{- \gamma(y)} dy$
for arbitrary constants $c \in \mathbb{R}$ and $d > 0$.
For instance, if $Y_t = B_t + b t$ for $b \in \mathbb{R} \setminus \crl{0}$, then
$I = \mathbb{R}$ and $s$ can be chosen as $s(x) = - {\rm sign}(b) e^{-2 b x}$.
Or if $(Y_t)$ is a Bessel process of dimension $\delta = 2(1-\nu) > 2$
starting at $y_0 > 0$, then one can choose $I = (0,\infty)$ and
$s(x) = - x^{2 \nu}$.

Denote the drawdown process $\overline{Y}_t - Y_t$ by $DD_t$ and if
$y_0 > 0$, the relative drawdown $DD_t/\overline{Y}_t$ by $rDD_t$.
As in Subsection \ref{subsecddm} we consider a stopping time of the form
$$
T_{\lambda} = \inf \crl{t : Y_t \le \lambda(\overline{Y}_t)}
$$
for a Borel function $\lambda$. But this time we assume that $\lambda$ maps
$[y_0,\infty) \cap I$ to the closure $\bar{I}$ of $I$ in $[-\infty,\infty]$
such that $\lambda(x) < x$ for all $x \in [y_0, \sup I)$.
Extend $s$ continuously to $s : \bar{I} \to [-\infty, \infty]$ and denote
\be \label{LaY}
g_{\lambda} = \sup \crl{t \le T_{\lambda} : Y_t = \overline{Y}_t}
\quad \mbox{and} \quad
\Lambda(x) = \int_{y_0}^x \frac{ds(y)}{s(y) - s \circ \lambda(y)}.
\ee
Then $\Lambda$ is a well-defined increasing function from $[y_0, \sup I]$ to $[0,\infty]$.
The following result generalizes Proposition \ref{proplambdastop}:

\begin{proposition} \label{proplambdastopY}
Let $(a_n)_{n \in \mathbb{N}}$ be an increasing sequence in
$(y_0,\sup I)$ and $(\varepsilon_n)_{n \in \mathbb{N}}$
a decreasing sequence in $(0,\infty)$ such that
$\lambda(x) \le x - \varepsilon_n$ for $y_0 \le x \le a_n$.
Denote $a = \lim_{n \to \infty} a_n \in (y_0, \sup I]$ and assume that
\be \label{condYTl}
\overline{Y}_{T_{\lambda}} < a \quad \mbox{and} \quad
Y_{T_{\lambda}} = \lambda(\overline{Y}_{T_{\lambda}}).
\ee
Then $\Lambda(a) = \infty$, $g_{\lambda} < T_{\lambda}$,
$Y_{T_{\lambda}} < \overline{Y}_{g_{\lambda}} = \overline{Y}_{T_{\lambda}} $ and
\be \label{repLDDY}
\p[g_{\lambda} \le T \mid {\cal F}_T] =
\frac{s(\overline{Y}_T) - s(Y_T)}{s(\overline{Y}_T)
- s \circ \lambda(\overline{Y}_T)}
\ee
for every stopping time $T \le T_{\lambda}$. Moreover, for all Borel functions
$h : [y_0,a) \to \mathbb{R}$ satisfying
$$
\int_{y_0}^a |h(y)| e^{- \Lambda(y)} d \Lambda(y) < \infty
$$
one has
\bea
\label{hedgingDDY}
\E{h(\overline{Y}_{T_{\lambda}}) \mid {\cal F}_T}
&=& h^{\Lambda}(y_0)
+ \int_0^T \frac{h^{\Lambda}(\overline{Y}_u)
- h(\overline{Y}_u)}{s(\overline{Y}_u) - s \circ \lambda(\overline{Y_u})} ds(Y_u)\\
\label{hDDY}
&=& \frac{h(\overline{Y}_T)[s(\overline{Y}_T)
- s(Y_T)] + h^{\Lambda}(\overline{Y}_T)[s(Y_T)
- s \circ \lambda(\overline{Y}_T)]}{s(\overline{Y}_T) - s \circ \lambda(\overline{Y_T})},
\eea
where
$$
h^{\Lambda}(x) = e^{\Lambda(x)} \int_x^a h(y) e^{- \Lambda(y)} d\Lambda(y),
\quad x \ge y_0.
$$
In particular,
\be \label{distbarY}
\p[\overline{Y}_{T_{\lambda}} > x \mid {\cal F}_T]
= 1_{\crl{\overline{Y}_T > x}} + 1_{\crl{\overline{Y}_T \le x}}
\frac{s(Y_T) - s \circ \lambda(\overline{Y}_T)}{s(\overline{Y}_T)
- s \circ \lambda(\overline{Y_T})}
e^{\Lambda(\overline{Y}_T) - \Lambda(x)} \quad \mbox{for } x \ge y_0.
\ee
\end{proposition}

\begin{proof}
$X_t = s(\overline{Y}_{t \wedge T_{\lambda}}) - s(Y_{t \wedge T_{\lambda}})$ is a non-negative
process of class $\cs$ starting at $0$ with
decomposition $(s(y_0) - s(Y_{t \wedge T_{\lambda}})) +
(s(\overline{Y}_{t \wedge T_{\lambda}}) - s(y_0))$.
The function $f : \mathbb{R}_+ \to \mathbb{R}_+$ given by
$$
f(x) = 1_{\crl{x < s(a) - s(y_0)}}
\frac{1}{x + s(y_0) - s \circ \lambda \circ s^{-1}(x + s(y_0))}
$$
satisfies the assumptions of Theorem \ref{thmfXg} with
$\tilde{a}_n = s(a_n) - s(y_0)$ and $\tilde{a} = s(a) - s(y_0)$
instead of $a_n$ and $a$.
Assumption \eqref{condYTl} guarantees that the process
$$
f(s(\overline{Y}_{t \wedge T_{\lambda}}) - s(y_0)) X_t
= 1_{\crl{\overline{Y}_{t \wedge T_{\lambda} < a}}} \frac{s(\overline{Y}_{t \wedge T_{\lambda}})
- s(Y_{t \wedge T_{\lambda}})}{s(\overline{Y}_{t \wedge T_{\lambda}})
- s \circ \lambda(\overline{Y}_{t \wedge T_{\lambda}})}
$$
takes values in $[0,1]$ and converges to $1$ almost surely.
So it follows from Theorem \ref{thmfXg} that $\Lambda(a) =
\int_0^{\tilde{a}} f(x) dx = \infty$,
$g_{\lambda} < T_{\lambda}$, $Y_{T_{\lambda}} < \overline{Y}_{g_{\lambda}}
= \overline{Y}_{T_{\lambda}}$
and
$$
\p[g_{\lambda} \le T \mid {\cal F}_T] =
\frac{s(\overline{Y}_T) - s(Y_T)}{s(\overline{Y}_T)
- s \circ \lambda(\overline{Y}_T)}
$$
for all stopping times $T \le T_{\lambda}$. Formulas \eqref{hedgingDDY}--\eqref{distbarY}
follow from Theorem \ref{thmfXg} applied to the function
$$
\tilde{h}(x) = h (s^{-1}(x+s(y_0))), \quad 0 \le x < \tilde{a}.
$$
\end{proof}

The generalization of Proposition \ref{propcond}
to the context of the current subsection looks as follows:

\begin{proposition} \label{propcondY}
Assume $(Y_t)$ is continuous and let $a \in (y_0, \sup I]$.
Then both of the following two conditions imply \eqref{condYTl}
\begin{itemize}
\item[{\rm (1)}]
$s(y_0) > 0$, $s(Y_t) \to 0$ almost surely, $Y_t < a$ for all $t$
and $s \circ \lambda(x) \ge 0$ for all $x \in [y_0,a)$
\item[{\rm (2)}]
$s(\overline{Y}_{\infty}) = \Lambda(a) = \infty$
\end{itemize}
\end{proposition}

\begin{proof}
Under assumption (1) one has $\overline{Y}_{\infty} < a$ and
$$
\frac{s(\overline{Y}_t) - s(Y_t)}{s(\overline{Y}_t) - s \circ \lambda (\overline{Y}_t)}
\ge \frac{s(\overline{Y}_t) - s(Y_t)}{s(\overline{Y}_t)} \to 1 \mbox{ almost surely.}
$$
It follows that $\overline{Y}_{T_{\lambda}} < a$ and
$Y_{T_{\lambda}} = \lambda(\overline{Y}_{T_{\lambda}})$.

If condition (2) holds, then $\int_0^{\infty} f(x) dx = \infty$ for the function
$$
f(x) = 1_{\crl{0 \le x < s(a) - s(y_0)}}
\frac{1}{x + s(y_0) - s \circ \lambda \circ s^{-1}(x + s(y_0))}.
$$
Since $X_t = s(\overline{Y}_t) - s(Y_t)$ is
a non-negative continuous process of class $\cs$ with decomposition
$X_t = (s(y_0) - s(Y_t)) + (s(\overline{Y}_t) -s(y_0))$,
it follows from Corollary \ref{corfiniteR} that
$\inf \crl{t : f(s(\overline{Y}_t) - s(y_0)) X_t \ge 1} < \infty$. This implies
$T_{\lambda} < \infty$, and \eqref{condYTl} follows.
\end{proof}

We now are ready to extend our results on maximum drawdowns and
maximum relative drawdowns of the last subsection.

\begin{proposition} \label{propmdd1Y}
Assume $(Y_t)$ is continuous with $s(y_0) > 0$ and $s(Y_t) \to 0$ almost surely.
Let $T < \infty$ be a stopping time and $K$ an
${\cal F}_T$-measurable random variable such that $0 \le s(K) < s(Y_T)$.
Denote $T_K = \inf \crl{t \ge T : Y_t = K}$. Then one has for all $x \ge 0$,
\be \label{mdd1Y}
\p \edg{\sup_{t \in [T, T_K] \cap \mathbb{R}_+} DD_t > x \mid {\cal F}_T}
= 1_{\crl{\overline{Y}_T - K > x}} + 1_{\crl{\overline{Y}_T - K \le x}}
\frac{s(Y_T) - s(K)}{s(K+x) - s(K)}.
\ee
If in addition, $y_0 > 0$, then
\be \label{mrdd1Y}
\p \edg{\sup_{t \in [T,T_K] \cap \mathbb{R}_+} rDD_t > x \mid {\cal F}_T}
= 1_{\crl{1 - K/\overline{Y}_T > x}} +
1_{\crl{1 - K/\overline{Y}_T \le x < 1}}
\frac{s(Y_T) - s(K)}{s(K/(1-x)) - s(K)}.
\ee
\end{proposition}

\begin{proof}
First assume that $T = 0$ and $K$ is a constant. Then $(Y_t)$ with
$\lambda \equiv K$ and $a = \sup I$ fulfills condition (1)
of Proposition \ref{propcondY}. Indeed, $s(y_0) > 0$, $s(Y_t) \to 0$ almost surely
and $s \circ \lambda(x) = s(K) \ge 0$ for all $x \in [y_0,a)$ are part of the
assumptions. To see that $Y_t < a$ for all $t$,
denote $R = \inf \crl{t : s(Y_t) = 0}$ and notice that $M_t = s(Y_{t \wedge R})$ is
a non-negative local martingale starting at $s(y_0) > 0$ and converging to zero
almost surely. Therefore, it follows from Doob's maximal identity that
$s(y_0)/s(\overline{Y}_R)$ is uniformly distributed on the interval $(0,1)$
(see Remark \ref{remDoobmax}).
In particular, $s(a) = \infty$ and hence, $Y_t < a$ for all $t$. It now follows
from Proposition \ref{propcondY} that the
conditions of Proposition \ref{proplambdastopY} are fulfilled. Moreover,
$$
\sup_{t \in [T,T_K] \cap \mathbb{R}_+} DD_t = \overline{Y}_{T_K} -K \quad \mbox{and} \quad
\sup_{t \in [T,T_K] \cap \mathbb{R}_+} rDD_t = 1 - K/\overline{Y}_{T_K}.
$$
So \eqref{mdd1Y} and \eqref{mrdd1Y} can be deduced from formula \eqref{distbarY}.
In the general case, the proposition follows
by considering the process $\tilde{Y}_t = Y_{T + t}$ in the filtration
$\tilde{\cal F}_t = {\cal F}_{T+t}$ and conditioning on ${\cal F}_T$.
\end{proof}

\begin{proposition} \label{propmdd2Y}
Assume $(Y_t)$ is continuous and $s(\overline{Y}_{\infty}) = \infty$.
Let $T < \infty$ be a stopping time and $K$ a $[0,\infty]$-valued ${\cal F}_T$-measurable
random variable such that $\overline{Y}_T < K \le \overline{Y}_{\infty}$.
Denote $T_K = \inf \crl{t \ge T : Y_t = K}$. Then one has for all $x \ge 0$,
\be \label{probYl}
\p[Y_t \ge \lambda(\overline{Y}_t)
\mbox{ for all } t \in [T, T_K] \cap \mathbb{R}_+ \mid {\cal F}_T]
= \brak{\frac{s(Y_T) - s \circ \lambda(\overline{Y}_T)}{s(\overline{Y}_T)
- s \circ \lambda(\overline{Y}_T)}}^+
e^{\Lambda(\overline{Y}_T) - \Lambda(K)}.
\ee
\end{proposition}

\begin{proof}
$X_t = s(\overline{Y}_t) - s(Y_t)$ is a non-negative process of class $\cs$
with decomposition $(s(y_0) - s(Y_t)) + (s(\overline{Y}_t) - s(y_0))$ and
$$
f(x) =
\frac{1}{x + s(y_0) - s \circ \lambda \circ s^{-1}(x + s(y_0))}
$$
a non-negative Borel function from
$[0, s(\sup I) - s(y_0))$ to $\mathbb{R}_+$. Since
$Y_t \ge \lambda(\overline{Y}_t)$ is equivalent to
$f(s(\overline{Y}_t) - s(y_0)) X_t \le 1$ and
$T_K = \inf \crl{t \ge T : s(\overline{Y}_t) -s(y_0) = s(K)-s(y_0)}$,
one obtains from Corollary \ref{corfiniteR} that
\beas
\p[Y_t \ge \lambda(\overline{Y}_t)
\mbox{ for all } t \in [T, T_K] \cap \mathbb{R}_+ \mid {\cal F}_T]
&=& \brak{\frac{s(Y_T) - s \circ \lambda(\overline{Y}_T)}{s(\overline{Y}_T)
- s \circ \lambda(\overline{Y}_T)}}^+
\exp \brak{- \int_{s(\overline{Y}_T) - s(y_0)}^{s(K) - s(y_0)} f(x) dx}\\
&=& \brak{\frac{s(Y_T) - s \circ \lambda(\overline{Y}_T)}{s(\overline{Y}_T)
- s \circ \lambda(\overline{Y}_T)}}^+ \exp \brak{\Lambda(\overline{Y}_T) - \Lambda(K)}.
\eeas
\end{proof}

\begin{corollary} \label{cormdd2Y}
If the assumptions of Proposition \ref{propmdd2Y} hold, then
\be \label{mdd2Y}
\p \edg{\sup_{t \in [T,T_K] \cap \mathbb{R}_+} DD_t \le x \mid {\cal F}_T}
= \brak{\frac{s(Y_T) - s(\overline{Y}_T- x)}{s(\overline{Y}_T) - s(\overline{Y}_T- x)}}^+
\exp \brak{- \int_{\overline{Y}_T}^K \frac{ds(y)}{s(y) - s(y-x)}}
\ee
for every constant $x > 0$ such that $y_0 - x \in \bar{I}$.
If in addition $y_0 > 0$, then
\be \label{mrdd2Y}
\p \edg{\sup_{t \in [T,T_K] \cap \mathbb{R}_+} rDD_t \le x \mid {\cal F}_T}
= \brak{\frac{s(Y_T) - s([1-x]\overline{Y}_T)}{s(\overline{Y}_T) - s([1-x]\overline{Y}_T)}}^+
\exp \brak{- \int_{\overline{Y}_T}^K \frac{ds(y)}{s(y) - s([1-x]y)}}
\ee
for each $x > 0$ such that
$\inf \crl{(1-x) y : y \in [y_0, \sup I)} \in \bar{I}$.
\end{corollary}

\begin{proof}
Formula \eqref{mdd2Y} follows from
Proposition \ref{propmdd2Y} applied to the function $\lambda(y) = y-x$.
The condition $y_0 - x \in \bar{I}$ ensures that
$\lambda([y_0, \infty)  \cap I)
\subset \bar{I}$. Formula \eqref{mrdd2Y} is obtained from
Proposition \ref{propmdd2Y} applied to the function
$\lambda(y) = (1-x)y$. $\inf \crl{(1-x) y : y \in [y_0, \sup I)} \in \bar{I}$
implies that $\lambda([y_0, \infty) \cap I) \subset \bar{I}$.
\end{proof}

\setcounter{equation}{0}
\section{Applications}
\label{secAppl}

We now discuss some applications in financial modelling and risk management.
We consider a stochastic process modelling the evolution of a financial asset.
In Subsection \ref{subsecprice} below it is assumed to be a
continuous local martingale $(M_t)$. In Subsection \ref{subsecrisk}
it will be a solution of an SDE of the form
$dY_t = \mu(Y_t) dt + \sigma(Y_t) dB_t$.

\subsection{Pricing and hedging of options on running maxima}
\label{subsecprice}

In standard mathematical finance it is usually assumed that there exists a risk-neutral measure
under which discounted prices of tradable assets are
non-negative local martingales. In the benchmark approach of Platen \cite{P}
(see also Platen and Heath \cite{PH}) prices are local martingales under
the physical probability measure
when expressed in terms of the benchmark portfolio. For our results to apply
we also need the price $(M_t)$ of an asset to be continuous and
$M_t \to 0$ almost surely as $t \to \infty$.
So we here assume that $(M_t)$ is a continuous non-negative local martingale
with constant initial value $m > 0$ such that $M_t \to 0$ almost surely.
We shall be interested in options that depend on the running maximum
of $(M_t)$ and are triggered by one of the following three
stopping times:\\
\hspace*{1cm} 1. $T_c = \inf \crl{t : M_t = c}$\\
\hspace*{1cm} 2. $T_c = \inf \crl{t : DD_t = c}$\\
\hspace*{1cm} 3. $T_c = \inf \crl{t : rDD_t = c}$\\
for a constant $c \ge 0$, where $DD_t = \overline{M}_t - M_t$
and $rDD_t = 1 - M_t/\overline{M}_t$.

\subsubsection{Options with downfall triggers}

Let us first consider an option with payoff of the form $h(\overline{M}_{T_c})$
for a Borel function $h : [m, \infty) \to \mathbb{R}$ and
$T_c = \inf \crl{t : M_t = c}$ for some constant $c \in [0,m)$, that is,
the option depends on the running maximum $(\overline{M}_t)$ and is triggered
the first time when $(M_t)$ drops to $c$. $T_c$ can be written as
$T_c = \inf \crl{t : M_t = \lambda(\overline{M}_t)}$ for $\lambda \equiv c$.
The functions $\Lambda$ and $h^{\Lambda}$ of Subsection \ref{subsecddm}
(see formulas \eqref{La} and \eqref{hLa})
corresponding to this particular $\lambda$ take the form
$$
\Lambda(x) = \log(x-c) - \log(m-c) \quad \mbox{and} \quad
h^{\Lambda}(x) = (x-c) \int_x^{\infty} \frac{h(y)}{(y-c)^2} dy.
$$
So it follows from Proposition \ref{proplambdastop} that
if $h$ satisfies the integrability condition
$\int_m^{\infty} |h(y)|/(y-c)^2 dy < \infty$, one has
\bea
\label{FDhedging}
\E{h(\overline{M}_{T_c}) \mid {\cal F}_{t \wedge T_c}}
&=& h^{\Lambda}(m)
+ \int_0^{t \wedge T_c}
\frac{h^{\Lambda}(\overline{M}_u) - h(\overline{M}_u)}{\overline{M}_u-c} dM_u\\
\label{FDprice}
&=& \frac{h(\overline{M}_{t \wedge T_c}) DD_{t \wedge T_c} +
h^{\Lambda}(\overline{M}_{t \wedge T_c})
(M_{t \wedge T_c}-c)}{\overline{M}_{t \wedge T_c} - c}.
\eea
Formula \eqref{FDhedging} provides a hedging strategy and
\eqref{FDprice} the fair price of the option at time $t \wedge T_c$.

\subsubsection{Options with drawdown triggers}

Now let the option payoff be given by $h(\overline{M}_{T_c})$, where $T_c$
is the first time the drawdown $(DD_t)$ hits some level $c \in (0,m]$. Then
$T_c$ can be written as $T_c = \inf \crl{t : M_t = \lambda(\overline{M}_t)}$
for the function $\lambda(y) = y-c$, and the functions $\Lambda$ and $h^{\lambda}$
of Subsection \ref{subsecddm} become
$$
\Lambda(x) = (x-m)/c \quad \mbox{and} \quad
h^{\Lambda}(x) = \frac{1}{c} e^{x/c} \int_x^{\infty} h(y) e^{- y/c} dy.
$$
So it follows from Proposition \ref{proplambdastop} that if $h$ satisfies
$\int_m^{\infty} |h(y)| e^{- y/c} dy < \infty$, one has
\bea
\label{DDhedging}
\E{h(\overline{M}_{T_c}) \mid {\cal F}_{t \wedge T_c}}
&=& h^{\Lambda}(m)
+ \frac{1}{c} \int_0^{t \wedge T_c}
(h^{\Lambda}(\overline{M}_u) - h(\overline{M}_u)) dM_u\\
\label{DDprice}
&=& h(\overline{M}_{t \wedge T_c}) \frac{DD_{t \wedge T_c}}{c} +
h^{\Lambda}(\overline{M}_{t \wedge T_c})
\brak{1 -\frac{DD_{t \wedge T_c}}{c}}.
\eea
Again, formula \eqref{DDhedging} gives the hedging strategy and
\eqref{DDprice} the fair price of the option.

\subsubsection{Options with relative drawdown triggers}

Consider an option with payoff $h(\overline{M}_{T_c})$ for the stopping time
$T_c = \inf \crl{t : rDD_t = c}$, where $c \in (0,1]$. Then
$T_c = \inf \crl{t : M_t = \lambda(\overline{M}_t)}$ for
$\lambda(y) = (1-c)y$. The functions $\Lambda$ and $h^{\Lambda}$ of
Subsection \ref{subsecddm} than take the form
$$
\Lambda(x) = \frac{1}{c} \log(x/m) \quad \mbox{and} \quad
h^{\Lambda}(x) =
\frac{1}{c} x^{1/c} \int_x^{\infty} h(y) y^{-(1+c)/c} dy,
$$
and Proposition \ref{proplambdastop} gives for all functions
$h$ satisfying $\int_m^{\infty} |h(y)|y^{-(1+c)/c} dy < \infty$,
\beas
\E{h(\overline{M}_{T_c}) \mid {\cal F}_{t \wedge T_c}}
&=& h^{\Lambda}(m)
+ \int_0^{t \wedge T_c} \frac{h^{\Lambda}(\overline{M}_u)
- h(\overline{M}_u)}{c \overline{M}_u} dM_u\\
&=& h(\overline{M}_{t \wedge T_c}) \frac{rDD_{t \wedge T_c}}{c} +
h^{\Lambda}(\overline{M}_{t \wedge T_c}) \brak{1 -\frac{rDD_{t \wedge T_c}}{c}},
\eeas
showing how to hedge and price the option.

\subsection{Risk management}
\label{subsecrisk}

Risk managers and regulators are typically interested in the distribution of
prices under the physical probability measure.
In standard mathematical finance they are assumed to follow semimartingales.
In the benchmark approach of Platen \cite{P}
they are local martingales. Let us here consider a price process
$(Y_t)$ taking values in an interval $I \subset \mathbb{R}$ and
satisfying an SDE of the form
$dY_t = \mu(Y_t)dt + \sigma(Y_t) dB_t$, $Y_0 = y_0 \in I$
for a Brownian motion $(B_t)$ such that
$$
\gamma(x) = 2 \int_{y_0}^x \frac{\mu(y)}{\sigma^2(y)} dy \quad \mbox{and}
\quad \int_{y_0}^x e^{- \gamma(y)} dy \quad
\mbox{are finite for all } x \in I.
$$
Choose constants $c \in \mathbb{R}$, $d >0$ and set
$s(x) = c + d \int_{y_0}^x e^{- \gamma(y)} dy$. Then
the scaled process $s(Y_t)$ is a local martingale. Let
$\lambda : [y_0, \infty) \cap I \to \bar{I}$
be a Borel function satisfying the assumptions of Proposition
\ref{proplambdastopY} for some $a \le \infty$ and denote
$$
T_{\lambda} = \inf \crl{t : Y_t = \lambda(\overline{Y}_t)} \quad \mbox{and} \quad
g_{\lambda} = \sup \crl{t \le T_{\lambda} : Y_t = \overline{Y}_t}.
$$
The function $\Lambda$ defined in \eqref{LaY} then becomes
$$
\Lambda(x) = \int_{y_0}^x \frac{e^{-\gamma(y)} dy}{\int_{\lambda(y)}^y e^{-\gamma(z)}dz},
$$
and one obtains from Proposition \ref{proplambdastopY}
that for all stopping times $T \le T_{\lambda}$,
$$
\p[g_{\lambda} \le T \mid {\cal F}_T]
= \frac{\int_{Y_T}^{\overline{Y}_T} e^{-\gamma(y)}dy}{
\int_{\lambda(\overline{Y}_T)}^{\overline{Y}_T} e^{-\gamma(y)}dy}
$$
and
\be \label{diffbarY}
\p[\overline{Y}_{T_{\lambda}} > x \mid {\cal F}_T]
= 1_{\crl{\overline{Y}_T > x}} + 1_{\crl{\overline{Y}_T \le x}}
\frac{\int_{\lambda(\overline{Y}_T)}^{Y_T} e^{- \gamma(y)} dy}{
\int_{\lambda(\overline{Y}_T)}^{\overline{Y}_T} e^{- \gamma(y)} dy}
\exp \brak{- \int_{\overline{Y}_T}^x
\frac{e^{-\gamma(y)}dy}{\int_{\lambda(y)}^y e^{-\gamma(z)}dz}}
\quad \mbox{for } x \ge y_0.
\ee
In the special case $\lambda \equiv c$, \eqref{diffbarY} reduces to
$$
\p[\overline{Y}_{T_{\lambda}} > x \mid {\cal F}_T]
= 1_{\crl{\overline{Y}_T > x}} + 1_{\crl{\overline{Y}_T \le x}}
\frac{\int^{\overline{Y}_T}_c e^{-\gamma(y)}dy}{\int_c^x e^{-\gamma(y)} dy}
\quad \mbox{for } x \ge y_0,
$$
and for $\lambda(y) = y - c$ it becomes
$$
\p[\overline{Y}_{T_{\lambda}} > x \mid {\cal F}_T]
= 1_{\crl{\overline{Y}_T > x}} + 1_{\crl{\overline{Y}_T \le x}}
\frac{\int_{\overline{Y}_T - c}^{Y_T} e^{- \gamma(y)} dy}{
\int_{\overline{Y}_T - c}^{\overline{Y}_T} e^{- \gamma(y)} dy}
\exp \brak{- \int_{\overline{Y}_T}^x
\frac{e^{-\gamma(y)}dy}{\int_{y-c}^y e^{-\gamma(z)}dz}}
\quad \mbox{for } x \ge y_0.
$$
For $T=0$, this gives
$$
\p[\overline{Y}_{T_{\lambda}} > x]
=  \exp \brak{- \int_{y_0}^x
\frac{e^{-\gamma(y)}dy}{\int_{y-c}^y e^{-\gamma(z)}dz}}
\quad \mbox{for } x \ge y_0,
$$
which (in the case $y_0=0$) is formula (3) of Lehoczky \cite{L}.

If $s(y_0) > 0$, $s(Y_t) \to 0$ almost surely and
there exists an ${\cal F}_T$-measurable
random variable $K$ such that $0 \le s(K)< s(Y_T)$, denote
$T_K = \inf \crl{t \ge T : Y_t = K}$. Then it follows from
Proposition \ref{propmdd1Y} that for all $x \ge 0$,
$$
\p \edg{\sup_{t \in [T, T_K] \cap \mathbb{R}_+} DD_t > x \mid {\cal F}_T}
= 1_{\crl{\overline{Y}_T - K > x}} + 1_{\crl{\overline{Y}_T - K \le x}}
\frac{\int^{Y_T}_K e^{-\gamma(y)} dy}{\int^{K+x}_K e^{-\gamma(y)}dy}.
$$
If in addition, $y_0 > 0$, then
$$
\p \edg{\sup_{t \in [T,T_K] \cap \mathbb{R}_+} rDD_t > x \mid {\cal F}_T}
= 1_{\crl{1 - K/\overline{Y}_T > x}} +
1_{\crl{1 - K/\overline{Y}_T \le x < 1}}
\frac{\int_K^{Y_T} e^{-\gamma(y)} dy}{\int^{K/(1-x)}_K e^{-\gamma(y)}dy}.
$$

On the other hand, if
$$
\int_{y_0}^{\overline{Y}_{\infty}} e^{-\gamma(y)} dy = \infty
$$
and $K$ is a $[0,\infty]$-valued ${\cal F}_T$-measurable
random variable such that $\overline{Y}_T < K \le \overline{Y}_{\infty}$ we denote
$T_K = \inf \crl{t \ge T : Y_t = K}$ and obtain from Corollary \ref{cormdd2Y}
that for all $x \ge 0$,
$$
\p \edg{\sup_{t \in [T,T_K] \cap \mathbb{R}_+} DD_t \le x \mid {\cal F}_T}
= \brak{\frac{\int_{\overline{Y}_T- x}^{Y_T} e^{-\gamma(y)}
dy}{\int_{\overline{Y}_T- x}^{\overline{Y}_T} e^{-\gamma(y)} dy}}^+
\exp \brak{- \int_{\overline{Y}_T}^K \frac{e^{-\gamma(y)} dy}{\int_{y-x}^y
e^{-\gamma(z)}dz}}
$$
for every constant $x > 0$ such that $y_0 - x \in \bar{I}$.
If in addition $y_0 > 0$, then
$$
\p \edg{\sup_{t \in [T,T_K] \cap \mathbb{R}_+} rDD_t \le x \mid {\cal F}_T}
= \brak{\frac{\int_{(1-x) \overline{Y}_T}^{Y_T} e^{-\gamma(y)} dy}{
\int_{(1-x) \overline{Y}_T}^{\overline{Y}_T} e^{-\gamma(y)} dy}}^+
\exp \brak{- \int_{\overline{Y}_T}^K \frac{e^{-\gamma(y)} dy}{\int_{(1-x)y}^y
e^{-\gamma(z)}dz}}
$$
for each $x > 0$ such that
$\inf \crl{(1-x) y : y \in [y_0, \sup I)} \in \bar{I}$.

\end{document}